\documentclass[12pt, reqno]{amsart}
\usepackage[T1]{fontenc}
\usepackage{dsfont}
\usepackage{mathrsfs}
\usepackage[colorlinks, citecolor=blue, linkcolor=blue]{hyperref}
\usepackage{xcolor}
\usepackage[a4paper,asymmetric]{geometry}
\usepackage{mathscinet}
\usepackage{latexsym}
\usepackage{amsthm}
\usepackage{amssymb}
\usepackage{amsfonts}
\usepackage{amsmath}
\usepackage{longtable}
\usepackage{graphicx}
\usepackage{multirow}
\usepackage{multicol}
\usepackage{amsfonts, amsmath}
\usepackage{latexsym,bm,amsfonts,amssymb,pifont,mathbbol,bbm}
\usepackage{extarrows}
\usepackage{verbatim}
\usepackage{tikz}
\DeclareMathAlphabet{\mathpzc}{OT1}{pzc}{m}{it}

\newtheorem{theorem}{Theorem}[section]

\newtheorem{lemma}[theorem]{Lemma}

\newtheorem{notation}[theorem]{Notation}
\newtheorem{proposition}[theorem]{Proposition}

\newtheorem{corollary}[theorem]{Corollary}

\theoremstyle{definition}

\theoremstyle{remark}
\newtheorem{remark}[theorem]{Remark}

\numberwithin{equation}{section}

\newcommand{\R}{\mathbb{R}}

\newcommand{\V}{\mathcal{V}}
\newcommand{\Hb}{\mathcal{H}}
\newcommand{\Lh}{\mathcal{L}}
\newcommand{\1}{\mathbf{1}}
\newcommand{\dif}{\mathrm{d}}
\newcommand\sgn{\mathrm{sgn}}
\newcommand{\Rnum}{\mathbb{R}}

\newcommand{\abs}[1]{\left\vert#1\right\vert}
\newcommand{\set}[1]{\left\{#1\right\}}
 \newcommand{\innp}[1]{\langle {#1}\rangle}

\date{}
\begin{document}
\title[Improved Berry-Ess\'{e}en Bound of LSE for fOU ]{An Improved Berry-Ess\'{e}en Bound of Least Squares Estimation for Fractional Ornstein-Uhlenbeck Processes\footnote{ This english version is translated by Hanxiao GENG from a Chinese version submitted to Acta Mathematica Scientia.}}

\author[Y. Chen]{Yong CHEN }
\address{Center for Applied Mathematics, School of Mathematics and Statistics, Jiangxi Normal University, 330022, Nanchang,  P. R. China} 
\email{zhishi@pku.org.cn}
\author[X. Gu]{Xiangmeng  {GU}}
\address{School of Mathematics and Statistics, Jiangxi Normal University, 330022, Nanchang,  P. R. China}
\begin{abstract}
 The aim of this paper is twofold. First, it offers a novel formula to calculate the inner product of the bounded variation function in the Hilbert space $\Hb$ associated with the fractional Brownian motion with Hurst parameter $H\in (0,\frac12)$. This formula is based on a kind of decomposition of the Lebesgue-Stieljes measure of the bounded variation function and the integration by parts formula of the Lebesgue-Stieljes measure. Second, as an application of the formula, we explore that as $T\to\infty$, the asymptotic line for the square of the norm of the bivariate function $f_T(t,s)=e^{-\theta|t-s|}1_{\{0\leq s,t\leq T\}}$  in the symmetric tensor space $\Hb^{\odot 2}$ (as a function of $T$), and improve the Berry-Ess\'{e}en type upper bound for the least squares estimation of the drift coefficient of the fractional Ornstein-Uhlenbeck processes with Hurst parameter $H\in (\frac14,\frac12)$. The asymptotic analysis of the present paper is much more subtle than that of Lemma 17 in Hu, Nualart, Zhou(2019) and the improved Berry-Ess\'{e}en type upper bound is the best improvement of the result of Theorem 1.1 in Chen, Li (2021). As a by-product, a second application of the above asymptotic analysis is given, i.e., we also show the Berry-Ess\'{e}en type upper bound for the moment estimation of the drift coefficient of the fractional Ornstein-Uhlenbeck processes where the method is obvious different to that of Proposition 4.1 in Sottinen, Viitasaari(2018). 
\end{abstract}

 \maketitle

	\section{Introduction}
	Unless otherwise specified, the Hurst parameter in this paper is always assumed to be $H\in(0,\frac12)$. This article has two main purposes. One is to improve the~Berry-Ess\'{e}en bound of the least squares estimation of the drift coefficients of the fractional~Ornstein-Uhlenbeck process based on continuous sample observations. The second is to give an easy to calculate formula for the inner product of Hilbert space $\Hb$ connected by fractional Brownian motion when it is restricted to bounded variation function. For the two purposes of this paper, the former can be regarded as a very effective application of the latter. In addition, as an accessory product, we also give the second application of the latter: The method of proving~Berry-Ess\'{e}en  bound for moment estimation of drift coefficients of fractional~Ornstein-Uhlenbeck process is different from that of Proposition 4.1 of \cite{SV 18} and Theorem 5.4 of \cite{Dou 22}. The conclusions of this paper are novel. It is particularly worth emphasizing that, as far as we know, there is no alternative method to obtain the upper bound of the improved ~Berry-Ess\'{e}en class for the least squares estimation of drift coefficients. In addition, we also point out that the binary function in symmetric tensor space $\Hb^{\odot 2}$ obtained by using this method
	\begin{equation}\label{ftst defn}   
	 f_T(t,s)=e^{-\theta|t-s|}1_{\{0\leq s,t\leq T\}}
	\end{equation}
	The asymptotic property of norm square (see Proposition~\ref{key1}) is much more precise than that of Lemma 17 in~\cite{HuNualartandZhou2019}.
    
    Specifically, we consider the fractional~Ornstein-Uhlenbeck process based on continuous time observation
		\begin{align}
		\dif X_t=-\theta X_t\dif t+\sigma\dif B_t^H,\,\,\,\,\,X_0=0,\,\,\,\,\,0\leq t\leq T,\label{OUED}
	\end{align}
	Berry-Ess\'{e}en class upper bounds for two kinds of estimators of drift coefficients, including~$\theta>0$ is the drift coefficient~$\sigma>0$ is the volatility coefficient, $B_t^H$ is the one-dimensional fractional Brownian motion with~Hurst parameter $H$, and its covariance function is given by the following formula:
	\begin{equation}\label{r_h}
		R_{H}(t,s)=\frac12(t^{2H}+s^{2H}-|t-s|^{2H}).
	\end{equation}

	Without losing generality, the following is constant $\sigma=1$. Reference~\cite{HuNualartandZhou2019} minimized the following formula
	\begin{align}
		\int^T_0|\dot{X}_t+\theta X_t|^2\dif t,\label{lse}
	\end{align}
and calculating the limit of the second moment of the sample (orbit) of the OU process
	\begin{equation}
	    \lim_{T\to\infty} \frac{1}{T} \int_0^T X_t^2\mathrm{d} t,
	\end{equation}
When traversal is constructed (i.e.~$\theta>0$), the least squares estimation and moment estimation of the drift coefficient are respectively:
	\begin{align}
		\hat{\theta}_T&=-\frac{\int^T_0X_t\dif X_t}{\int^T_0X_t^2\dif t}=\theta-\frac{\int^T_0X_t\dif B^H_t}{\int^T_0X_t^2\dif t},\label{Lse}\\
		\tilde{\theta}_{T}&=\Big( \frac{1}{ {H} \Gamma(2H ) T} \int_0^T X_t^2\mathrm{d} t \Big)^{-\frac{1}{2H}}.\label{theta tilde formula}
	\end{align}
	
    As in reference~\cite{HuNualartandZhou2019}, this paper does not discuss the meaning of the first random integral about the fractional OU process $X_t$ at the right end of \eqref{Lse}, but only regards it as a formal integral, that is, it is only understood as substituting the direct form of equation \eqref{OUED} into the integral. The second random integral about $B_t^H$ at the right end of \eqref{Lse} obtained after substitution is understood as a divergent (or skorohold) integral about fractional Brownian motion, but its meaning as a statistic in the sense of standard statistics is not studied. Of course, the statistical meaning of the second statistic moment estimation is completely clear.
	
	Further, by verifying the fourth order moment theorem, reference \cite{HuNualartandZhou2019} gives the strong convergence and asymptotic normality of the least squares estimate and the moment estimate. The two asymptotic properties of the norm of the binary function $f_T(t,s)$ and its contraction are the key steps. For the former, they use a formula of the inner product of space $\Hb$ and tensor space $\Hb^{\otimes 2}$, see \eqref{inner product_1} for details. This formula is the expression formula for the inner product of bounded variation function in Hilbert space associated with the general second moment process given by the integral by parts formula in combination with \cite{Jolis}: the inner product is equal to the integral of the product of the covariance function of the second moment process with respect to the measure derived from two bounded variation functions. For the norm of the compression of binary function $\frac{1}{\sqrt T}f_T(t,s)$, they use Fourier transform to prove that it tends to zero, see \eqref{flybh njff}.

    Based on the above results in \cite{HuNualartandZhou2019}, reference~\cite{chenandli2021} gives the convergence rate between the distribution of the least squares estimate and its asymptotic distribution, that is, the upper bound of~Berry-Ess\'een class: when $H\in (0,\frac12)$ and $T$ are sufficiently large, random variable
    $$\sqrt{T}(\hat{\theta}_T-\theta)$$
    and the upper bound of Kolmogorov distance of normal random variable is~$T^{-\beta}$, we have:
	\begin{align}
		\beta=
		\begin{cases}
		   \frac12, & H\in[0,\frac14],\\
			1-2H, & H\in(\frac14,\frac12).
		\end{cases}\label{beta}
	\end{align}

Here, the method of proving the Berry-Ess\'een bound of the least squares estimate is based on the Corollary~1 of \cite{Kim3} and two asymptotic analyses of the binary function $f_T(t,s)$. The method to prove the Berry-Ess\'een bound of moment estimation is to transform the fourth order moment into the two asymptotic analyses of the binary function $f_T(t,s)$ through the multiplication formula of multiple Wiener integrals, and to estimate the inner product of $f_T(t,s)$ and $h_T(t,s)$ (see \eqref{ht ts}), see \cite{CZ 21} and \cite{Chen guli 22}. Different from this, Proposition 4.1 of \cite{SV 18} and Theorem 5.4 of \cite{Dou 22}, the proof of the Berry-Ess\'een bound for moment estimation is to transform the fourth-order moment into an asymptotic analysis of the stationary solution of the fractional OU process by the Wick formula, and the latter is known, see \cite{Cheridito 03}.

    Review (\ref{beta}), when $H=\frac12-\varepsilon$ and $\varepsilon$ sufficiently small, $\beta$ tends to zero. This is the same as when $H=\frac12$, the known Berry-Ess\'een bound of $\sqrt{T}(\hat{\theta}_T-\theta)$ is $\frac{1}{\sqrt{T}} $, which is very far away, so a reasonable guess is:
\begin{center}\label{caice 1}
    ``when $H\in (\frac14, \frac12)$, 
    the upper bound of Berry-Ess\'{e}en class is still $\frac{1}{\sqrt{T}}. $''
\end{center} 
We will prove this conjecture in this paper. It can be seen from the proof of Theorem 1.1 in \cite{CZ 21} that the key problem is a more precise asymptotic analysis of the norm of the bivariate function $f_T(t,s)$. Therefore, obtaining the asymptotic analysis of this binary function norm is the key step of this paper, and we describe the result of this asymptotic analysis as the following theorem:

\begin{theorem}\label{qq1}
Let~$\theta>0,\,H\in(0,\frac12)$. For binary function $f_T(t,s)$ in space ${\Hb}^{\otimes2}$, see \eqref{ftst defn}, There is a normal number $C_{H,\theta}$ that does not depend on $T$, so that when $T$ is sufficiently large, there is an inequality
	\begin{align}\label{main}
	\Big|	\|f_T\|^2_{{\Hb}^{\otimes2}}-2(H\Gamma(2H))^2\sigma_H^2 T\Big|\le  C_{H,\theta}
		\end{align}
	holds, where
	\begin{equation}\label{sigmah2}
	    \sigma_H^2=(4H-1)+\frac{2\Gamma(2-4H)\Gamma(4H)}{\Gamma(2H)\Gamma(1-2H)}.
	\end{equation} 
\end{theorem}
\begin{remark}
\begin{enumerate}
\item The upper bound given by formula \eqref{main} in this paper is a constant $C_{H,\theta}$, which is independent of $T$, that is, the order of the upper bound of $T$ is 0. In contrast, the upper bound corresponding to the result of Lemma 3.11 in \cite{chenandli2021} is $T^{2H}$, that is, the order of the upper bound of $T$ is $2H$. Furthermore, the upper bound in this paper is the best upper bound in the sense of the asymptote below.
    \item In fact, the conclusion obtained in this paper is stronger than \eqref{main}. That is, this paper actually obtains the square of the norm of the binary function $f_T(t,s)$, as a function of $T$, the asymptote when $T\to \infty$:
	\begin{align}
		\lim_{T\rightarrow\infty}
		\big(\|f_T\|_{{\Hb}^{\otimes2}}^2-2(H\Gamma(2H))^2\sigma_H^2 T\big)={C}_H,  \label{jianjinxian mubiao}
	\end{align}
    Here ${C}_H\in\R$ is a constant that depends only on~$H$ and is independent of $T$. See the proof of Theorem 1.1 in Section 3 of this paper for details. We emphasize that in this paper, the intercept term ${C}_H$ of the asymptote is irrelevant, while the existence and slope of the asymptote play a key role.
    \item The standard $o, O$ symbols in asymptotic analysis are used to compare Lemma 17 in \cite{HuNualartandZhou2019}, Lemma 3.11 in \cite{chenandli2021}, and the formula \eqref{main} in this paper as follows: when $T\to\infty$, we have:
	\begin{align}
	    \frac{1}{T}\|f_T\|^2_{{\Hb}^{\otimes2}}-2(H\Gamma(2H))^2\sigma_H^2&= o(1),\label{zhou jielun}\\
	    \frac{1}{T}\|f_T\|^2_{{\Hb}^{\otimes2}}-2(H\Gamma(2H))^2\sigma_H^2&= O(T^{2H-1}),\label{chenli jielun}\\
	    \frac{1}{T}\|f_T\|^2_{{\Hb}^{\otimes2}}-2(H\Gamma(2H))^2\sigma_H^2&= O(T^{-1}).\label{benwen jielun}
	\end{align}
	As a comparison, the method of Lemma 17 in \cite{HuNualartandZhou2019} can obtain formula \eqref{zhou jielun} succinctly, and the method of Lemma 3.11 in \cite{chenandli2021} is based on Lemma 17 in \cite{HuNualartandZhou2019}. However, this method cannot be further improved, that is, the above formula \eqref{benwen jielun} cannot be obtained. In other words, the method of using the new formula for calculating the inner product of fractional Brownian motion given in this paper is, as far as we know, still irreplaceable.
	\end{enumerate}
\end{remark}
Starting from the asymptotic analysis given by the above theorem, the following theorem shows that when $H\in(\frac14,\frac12)$, the improved Berry-Ess\'een bound of the least squares estimate is the $\frac{1}{\sqrt{T}}$ guessed above, and as a by-product, the Berry-Ess\'een bound of the moment estimate is also $\frac{1}{\sqrt{T}} $:

\begin{theorem}\label{zhuyaodl 2}
Let $Z$ be a standard normal random variable and $H\in(0,\frac12)$. Then there is a normal number $C_{\theta, \,H}$, and it does not depend on $T$, so that when $T$ is large enough, there is \textnormal{Berry--Ess\'{e}en} inequality
\begin{align}
    \sup_{z\in \Rnum}\abs{P(\sqrt{\frac{T}{\theta \sigma^2_H}} (\hat{\theta}_T-\theta )\le z)-P(Z\le z)}&\le\frac{ C_{\theta, H}}{\sqrt{ T} }; \label{zuixiaoerch B_E }\\ 
\sup_{z\in \Rnum}\abs{P(\sqrt{\frac{4H^2 T}{\theta \sigma^2_{H}}} (\tilde{\theta}_T-\theta )\le z)-P(Z\le z)}&\le\frac{ C_{\theta, H}}{\sqrt{ T} };\label{juguji de B_E }
\end{align}
hold, where $\sigma^2_{H}$ as \eqref{sigmah2}.
\end{theorem}
	\begin{remark}
	\begin{enumerate}


We point out that in the \textnormal{Berry--Ess\'{e}en} class inequality estimates of two statistics (see \eqref{BE jie guji 1} and \eqref{BE jie guji 2}), part of the source of the upper bound $\frac{1}{\sqrt{T}}$ is based on the key inequality (3.17) in \cite{HuNualartandZhou2019}, that is, the upper bound of the function $f_T(s,t)$ about its own compressed $f_T\otimes_1f_T$norm in space $\Hb^{\otimes 2}$ is the key fact of $\sqrt{T}$. The upper bound estimation is obtained by another formula for calculating the inner product in $\Hb$, namely \textnormal{Fourier} transform, as shown in formula \eqref{flybh njff}. In a word, we finally get the \textnormal{Berry--Ess\'{e}en} class upper bound estimates of the two statistics in this paper using four very different formulas for calculating the inner product of $\Hb$: \eqref{inner product_2}, \eqref{neijibiaochu 0}, \eqref{inner product_1} and \eqref{flybh njff}. In other words, except for the formula \eqref{suanzi bianhuanfa} for calculating the inner product using the operator $K^*_H$, all the other four formulas for calculating the inner product of H mentioned in Section \ref{zhunb 1} have been used.
\end{enumerate} 
	\end{remark}
In this paper, the method of proving \textnormal{Berry--Ess\'{e}en} inequality \eqref{juguji de B_E } of moment estimation is based on the following proposition, which gives the estimation of the inner product of binary functions $f_T,\, h_T$ in tensor space ${\Hb}^{\otimes 2}$. Here, binary function	
\begin{align}
  h_T(t,s)&= e^{- {\theta(T-t)-\theta(T-s)}}\mathbb{1}_{\set{0\le s,t\le T}}.\label{ht ts} 
\end{align}
\begin{proposition}\label{jiaochaxiang}
 Let the binary functions $f_T,\, h_T$ be given in \eqref{ftst defn} and \eqref{ht ts} respectively, then there is a constant $C_H$ independent of $T$, which makes the following inequality hold:
    \begin{align}
        \abs{\innp{f_T,\, h_T}_{{\Hb}^{\otimes 2}}}\le C_H. \label{touyige budsh}
    \end{align}
\end{proposition}
\begin{remark}
Proposition \ref{jiaochaxiang} and Theorem \ref{qq1} have the same point in that they are both the inner product of two bivariate functions estimated in $\Hb^{\otimes 2}$. The difference is that the former actually divides the integral region into nine blocks, and finally reduces it to three kinds of integral calculations by symmetry and other methods. The latter takes advantage of the particularity of the function $h_T(s,t)$ to separate variables, so it regresses to the problem of estimating the inner product of two univariate functions in $\Hb$, and conveniently uses the inner product calculation formula in Inference \ref{coro-gu-3}. Compared with the two methods, the whole process of the former is very complicated and the latter is very simple. However, since the function $f_T(s,t)$ is not variable separated, the latter method is not applicable to the former. As far as we know, we do not know whether there are other simpler methods to prove the conclusion of Theorem \ref{qq1}.
\end{remark}

    In the second half of this section, we give a new formula for calculating the inner product of $H\in (0,\frac12)$ space-time $\Hb$ and symmetric tensor space $\Hb^{\odot 2}$. See Propositions \ref{key1} and \ref{key2}. The new formula is similar to but also obviously different from the following well-known facts to some extent: When Hurst parameter~$H\in(0,\frac12)$, the formula for the inner product of two disjoint functions $f$ and $g$ in~Hilbert space $\Hb$ connected by fractional Brownian motion is the same as that for the inner product when $H\in(\frac12,1)$, see \cite{Cheridito 03, Mishura 08}, or see Corollary ~\ref{coro-gu-3}. The new formula for calculating the inner product given in Proposition \ref{key1} of this paper can be explained as follows:  
The integral region $[0,T]^2$ is divided into the following three parts. For the double integral on region
\begin{equation}\label{kp 1}
    \kappa_1:=\set{(u,v)\in[0,T]^2:\,0\le v\le u-1\le T-1  }
\end{equation} and \begin{equation}\label{kp 2}
  \kappa_2:=\set{(u,v)\in[0,T]^2:\, 0\le u\le v-1\le T-1 } 
\end{equation} the Partial integral formula on the measure is applied twice, while for the double integral on region \begin{equation}\label{kp 3}
    \kappa_3:=\set{(u,v)\in[0,T]^2:\, 0\vee (u-1)\le v\le (u+1)\wedge T}
\end{equation} the Partial integral formula on the measure is applied only once. The integral domain decomposition is shown in Figure 1.
\begin{figure}
	\centering
	\vspace{-3.618mm}
		\setlength{\abovecaptionskip}{-0.618mm}
	\begin{tikzpicture}
		\draw (-2,2) -- (2,2);
		\draw (2,2) -- (2,-2);
		\draw[->] (-2.5,-2) -- (2.5,-2);
		\draw[->] (-2,-2.2) -- (-2,2.5);
		\node[left] at (-2,2) {$T$};
		\node[below] at (2,-2) {$T$};
		\node[above right] at (-2,2.2) {$v$};
		\node[above right] at (2.2,-2) {$u$};
		\draw (-1.5,-2) -- (2,1.5);
		\draw (-2,-1.5) -- (1.5,2);
		\node[centered] at (-0.798,0.906) {$\kappa_2$};
		\node[centered] at (0,0) {$\kappa_3$};
		\node[centered] at (0.906,-0.798) {$\kappa_1$};
	\end{tikzpicture}
	\caption{$\Hb$ schematic diagram of integral domain decomposition 
	in new calculation formula of inner product}\label{tu'}
\end{figure}
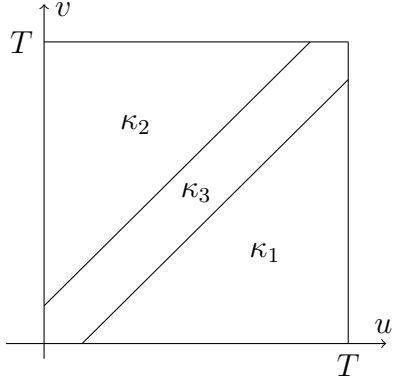

\begin{notation}\label{jihao 001}
Record $\alpha_H=H(2H-1)$. Let $\V_{[0,T]}$ be the whole set of bounded variation functions defined on~$[0,T]$. For any $f\in \V_{[0,T]}$, $f^0$ is defined as
		\begin{align}
			f^0(x) = \left\{
			\begin{array}{rl}
				f(x) , &\text{} \,x\in[0,T],\\
				0   ,  & \text{other} .
			\end{array} \right.\notag
		\end{align}
Record $\nu_f$ is the limit of \textnormal{Lebesgue-Stieljes} measure on $\big(\Rnum,\mathcal{B}(\Rnum)\big)$ of ~$f^0(x)$ connection on $\big([0,T],\mathcal{B}([0,T]\big)$. In particular, the following more special form is used in this paper, that is, let $0\le a<b\le T$, and $g=f\cdot \1_{[a,b]}$, where $f$ is a differentiable function, then:
\begin{equation}\label{jieshi01}
  \nu_g(\dif x)=  f'(x)\cdot \1_{[a,b]}(x)\dif x+  f(x)\cdot \big(\delta_a(x)-\delta_b(x)\big)\dif x,
\end{equation}
Here $\delta_a(\cdot)$ is a dirac generalized function whose mass is concentrated at point $a$. For ease of use, we use the notation partial $\frac{\partial g}{\partial x}$ to represent the ``density function'' in the form of measure \eqref{jieshi01}.
\end{notation}
\begin{remark} The details of the above measures can be found in~\cite{Jolis}, which is the source of the new inner product expression formula in this paper, and is also one of the starting points of this paper. The purpose of introducing this measure is to use the Partial integral formula about this measure. In other words, its convenience is to absorb the values of endpoints $a,\,b$ into the measure through two Dirac generalized functions (or Dirac single point measure) $\nu_g$, so that it is convenient to use the Partial integral formula of $\nu_g$. See Lemma \ref{partial integral} for details.
\end{remark}
	\begin{proposition}\label{key1}
		If~$f,g\in\V_{[0,T]}$, then
		\begin{small}
		\begin{align}
			\langle f,g\rangle_{\Hb}=&\alpha_H\Big(\int^T_{1} g(t) \dif t\int^{t-1}_0 f(s) (t-s)^{2H-2}\dif s +\int^T_{1} f(s) \dif s\int^{s-1}_0 g(t) (s-t)^{2H-2}\dif t\notag\Big)\notag\\
			&-H\int^T_0g(t)\dif t\int_{0}^{T}\big(t^{2H-1}-\mathrm{sgn}(t-s)|t-s|^{2H-1}\big)\nu_{\tilde{f}_t}(\mathrm{d}s)\label{inner product_2},
		\end{align}
		\end{small}
		including $\tilde{f_t}(s)=f(s)\cdot\1_{[(t-1)\vee0,(t+1)\wedge T]}(s)$) is a family of functions with~$s$ as the independent variable and~$t$ as the parameter. The meaning of $\nu_{\tilde{f}_t}(\mathrm{d}s)$) is given in notation \ref{jihao 001} and \eqref{jieshi01}. In addition, take any two positive numbers~$\varepsilon_1,\varepsilon_2\in(0,T)$, record
		$$\bar{f}_t(s)=f(s)\cdot\1_{[(t-\varepsilon_1)\vee0,(t+\varepsilon_2)\wedge T]}(s),$$
	then \eqref{inner product_2} can be generalized as: 
	\begin{small}
		\begin{align}
			\langle f,g\rangle_\Hb = &\alpha_H\Big(\int^T_{\varepsilon_1} g(t) \dif t\int^{t-\varepsilon_1}_0 f(s) (t-s)^{2H-2}\dif s +\int^T_{\varepsilon_2} f(s) \dif s\int^{s-\varepsilon_2}_0 g(t) (s-t)^{2H-2}\dif t\notag\Big)\notag\\
			&-H\int^T_0g(t)\dif t\int_{0}^{T}\big(t^{2H-1}-\mathrm{sgn}(t-s)|t-s|^{2H-1}\big)\nu_{\bar{f}_t}(\mathrm{d}s)\label{inner product_4}.
		\end{align}
	\end{small}
	\end{proposition}

\begin{remark}
The significant difference between the inner product calculation formula (Proposition \ref{key1}) obtained by the above division of the integral region and the inner product calculation formula when the supports do not intersect is that the latter requires that the support set of the binary function is $\set{(u,v):\,0\le v\le u \le T }$ or $\set{(u,v):\,0\le v\le u \le T }$ a rectangular sub region. See Corollary~\ref{coro-gu-3}. This significant difference is reflected in Figure 1 to some extent.
\end{remark}
	 Note that $\Hb^{\otimes2}$ and $\Hb^{\odot2}$ are quadratic tensor product spaces and quadratic symmetric tensor product spaces of~$\Hb$ . Proposition~\ref{key2} will give the calculation formula of the inner product of binary symmetric functions in ~$\Hb^{\odot2}$, which is the direct inference of Proposition \ref{key1}, so the derivation details are omitted below. For the convenience of expression, the following marks are introduced:
	\begin{notation}\label{jiaohao suanzi}
		Let $\Lh(C_0,\R)$ be the whole set of bounded linear functionals defined on the set of compact supported continuous functions $C_0$. Let $a,b\in [0,T]$, define three linear operators as follows:
		\begin{enumerate}
	\item  Operators of $\V_{[0,T]}\rightarrow \Lh(C_0,\R)$:
	\begin{equation*}
	 \dfrac{\partial_a}{{\partial} s}f(s) =  \dfrac{\partial}{{\partial} s}\big(f{(s)}\cdot\1_{[(a-1)\vee0,(a+1)\wedge T]}(s)\big),
    \end{equation*} I.e.: $\dfrac{\partial_a}{{\partial} s}f(s) $ is the density function of measure $\nu_{\tilde{f}_a}$ corresponding to function $\tilde{f_a}(s)$ in Proposition \ref{key1}, see also \eqref{jieshi01}.
		\item  Operators of $\V_{[0,T]}^{\otimes 2}\rightarrow \Lh(C_0,\R)\otimes \V_{[0,T]}$: \begin{align*}
			\dfrac{\partial_a}{{\partial} s}\varphi{(s,t)}=\dfrac{\partial}{{\partial} s}\big(\varphi{(s,t)}\cdot\1_{[(a-1)\vee0,(a+1)\wedge T]}(s)\big).
		\end{align*}
	\item Operators of $\V^{\otimes2}_{[0,T]}\rightarrow\Lh(C_0,\R)^{\otimes2}$: 
		\begin{align*}
			\dfrac{\partial_a\partial_b}{{\partial} s\partial t}\varphi{(s,t)}=\dfrac{\partial^2}{{\partial} s\partial t}\big(\varphi{(s,t)}\cdot\1_{[(a-1)\vee0,(a+1)\wedge T]}(s)\cdot\1_{[(b-1)\vee0,(b+1)\wedge T]}(t)\big),
		\end{align*} i.e. $\dfrac{\partial_a\partial_b}{{\partial} s\partial t}\varphi{(s,t)}$ is the \textnormal{Lebesgue-Stieljes} measure density function, which connected by the binary function $\varphi{(s,t)}\cdot\1_{[(a-1)\vee0,(a+1)\wedge T]}(s)\cdot\1_{[(b-1)\vee0,(b+1)\wedge T]}(t) $ on the space \begin{small} $$\Big( [(a-1)\vee0,(a+1)\wedge T]\times[(b-1)\vee0,(b+1)\wedge T],\,\mathcal{B}\big( [(a-1)\vee0,(a+1)\wedge T]\times[(b-1)\vee0,(b+1)\wedge T]\big)\Big)$$\end{small}
\end{enumerate}
	\end{notation}

	\begin{proposition}\label{key2}
		Let $\vec{s}=(s_1,s_2)$, $\vec{t}=(t_1,t_2)$ and~$(\vec{s},\vec{t} )\in \kappa_i\times \kappa_j,\, i,j=1,\, 2,\,3$, $\kappa_i$ see \eqref{kp 1}-\eqref{kp 3}. if~$\phi,\psi\in\V_{[0,T]}^{\odot2}$, then
		\begin{align}
			\begin{split}
			\langle\psi,\phi\rangle_{\Hb^{\otimes2}}=&\alpha_H^2\sum_{i,j=1}^2\int_{{\kappa}_i\times {\kappa}_j}\psi(s_1,t_1)\phi(s_2,t_2)|s_1-s_2|^{2H-2}|t_1-t_2|^{2H-2}\dif \vec{s}\dif \vec{t}\\
				&-2\alpha_H\sum_{i=1}^2\int_{{\kappa}_3\times {\kappa}_i}\psi(s_1,t_1)\frac{\partial_{s_1}}{\partial s_2}\phi(s_2,t_2)\frac{\partial R_H}{\partial s_1}(s_1,s_2)|t_1-t_2|^{2H-2}\dif \vec{s}\dif \vec{t}\\
				&+\int_{{\kappa}_3\times {\kappa}_3}\psi(s_1,t_1)\frac{\partial R_H}{\partial s_1}(s_1,s_2)\frac{\partial R_H}{\partial t_1}(t_1,t_2)\frac{\partial_{s_1}\partial_{t_1}}{\partial s_2\partial t_2}\phi(s_2,t_2)\dif \vec{s}\dif \vec{t},   \label{key2'}
			\end{split}
		\end{align}
 where $R_H(t_1,t_2) $ is the covariance of fractional Brownian motion(see \eqref{r_h}), 
Operators $\dfrac{\partial_a}{{\partial} s},\, \dfrac{\partial_a\partial_b}{{\partial} s\partial t}$ see mark \ref{jiaohao suanzi}.
	\end{proposition}
	\begin{remark}\label{difference}
 \begin{enumerate}
 	\item If~$\psi,\phi$ is asymmetric, then
		\begin{align}
			&\sum_{i=1}^2\int_{\kappa_3\times \kappa_i}\psi(s_1,t_1)\frac{\partial_{s_1}}{\partial s_2}\phi(s_2,t_2)\frac{\partial R_H}{\partial s_1}(s_1,s_2)|t_1-t_2|^{2H-2}\dif \vec{s}\dif \vec{t}\notag\\
			&\neq\sum_{i=1}^2\int_{\kappa_i\times \kappa_3}\psi(s_1,t_1)\frac{\partial_{t_1}}{\partial t_2}\phi(s_2,t_2)\frac{\partial R_H}{\partial t_1}(t_1,t_2)|s_1-s_2|^{2H-2}\dif \vec{s}\dif \vec{t}\notag.
		\end{align}
    \item The essence of inner product formula \eqref{key2'} is also measure decomposition, that is, for any given $(s_1,t_1)\in [0,T]^2$, $\nu_{\phi}$, the measure $\nu_{\phi}$ on $[0,T]^2$ associated with the binary function $\phi(s_2,t_2)\in \V_{[0,T]}^{\odot2}$ is decomposed into the sum of the measures derived from the limitation of $\phi(s_2,t_2)$ itself on $\big(\overline{M_{ij}},\mathcal{B}(\overline{M_{ij}})\big),\,\,i,j=1,2,3,$ (see Figure 2 for details), where
	\begin{align*}
	    M_{11}&=\set{(s_2,t_2)\in [0,T]^2,\, s_2\le s_1-1,\,t_2\le t_1-1},\\
	    M_{33}&=\set{(s_2,t_2)\in [0,T]^2,\, s_1-1<s_2\le s_1+1,\,t_1-1<t_2\le t_1+1},
	\end{align*} Others are similar and here $\overline{M_{ij}}$ is the closure of $M_{ij}$.
	\begin{figure}
		\centering
		\vspace{-3.618mm}
		\setlength{\abovecaptionskip}{-0.618mm}
		\begin{tikzpicture}
			\coordinate (A) at (0.598,0.526);
			\fill (A) circle[radius=2pt];
			\node at (2.6,1.226) {$(s_1,t_1)$};
			\node[left] at (-2,2) {$T$};
			\node[below] at (2,-2) {$T$};
			\node[above right] at (-2,2.0) {$t_2$};
			\node[above right] at (2.2,-2) {$s_2$};
			\node[below] at (A) {$M_{33}$};
			\node[below] at (-0.798,0.526) {$M_{13}$};
			\node[below] at (1.6,0.526) {$M_{23}$};
			\node[below] at (0.598,0.526-1.2) {$M_{31}$};
			\node[below] at (-0.798,0.526-1.2) {$M_{11}$};
			\node[below] at (1.6,0.526-1.2) {$M_{21}$};
			\node[below] at (0.598,0.526+1.2) {$M_{32}$};
			\node[below] at (-0.798,0.526+1.2) {$M_{12}$};
			\node[below] at (1.6,0.526+1.2) {$M_{22}$};
			\draw[dotted] (2.05,1.146) -- (A.west);
			\draw (-2,2) -- (2,2);
			\draw (2,2) -- (2,-2);
			\draw[->] (-2.5,-2) -- (2.9,-2);
			\draw[->] (-2,-2.2) -- (-2,2.38);
			\draw (0.598+0.55,-2) -- (0.598+0.55,2);
			\draw (0.598-0.55,-2) -- (0.598-0.55,2);
			\draw (-2,0.526+0.55) -- (2,0.526+0.55);
			\draw (-2,0.526-0.55) -- (2,0.526-0.55);
			\node[below] at (0.598-0.55,-2) {$s_1-1$};
			\node[below] at (0.598+0.55,-2) {$s_1+1$};
			\node[left] at (-2,0.526-0.55) {$t_1-1$};
			\node[left] at (-2,0.526+0.55) {$t_1+1$};
		\end{tikzpicture}\caption{Schematic diagram of measure decomposition}\label{tu}
	\end{figure}
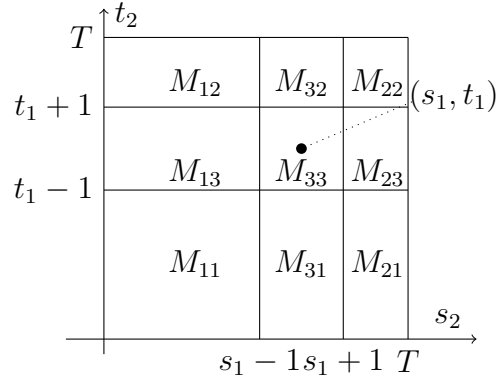
		\end{enumerate}
	\end{remark}
	
    The rest of this paper is arranged as follows: in Section \ref{zhunb 1} we briefly review various known calculation formulas of inner product in $\Hb$ and prove Proposition \ref{key1}. In Section \ref{sec three} we prove Proposition \ref{jiaochaxiang}, Theorem \ref{qq1} and Theorem \ref{zhuyaodl 2}. As an appendix, in Section \ref{jianjingfenxi} we give asymptotes of various multiple integrals used in the proof of Theorem \ref{qq1}. Finally, we point out that the constants $C_{H},\,C_{H,\theta}$ independent to $T$, and can be different from line to line.

    \section {Preparation knowledge and proof of new calculation formula of inner product in $\Hb$}\label{zhunb 1}
	\subsection{Preparation knowledge}
	    Record $\mathcal{E}$ as all the real valued ladder functions on $[0,T]$, and assign the inner product on them:
		\begin{equation}
			\langle\1_{[a,b)},\1_{[c,d)}\rangle_{\Hb}=\mathbb{E}\big((B_b^H-B_a^H)(B_d^H-B_c^H)\big).\label{HB}
		\end{equation}
	$\Hb$ is the \textnormal{Hilbert} space of~$\mathcal{E}$ after completion. On the premise of preserving the linear structure and norm, the mapping ~$\1_{[0,t]}\mapsto B_t^H$ is extended to~$\Hb$, and the isometric isomorphic mapping is recorded as~$\varphi\mapsto B^H(\varphi)$. And called $\{B^H(\varphi),\varphi\in\Hb\}$ is a Gaussian equidistant process connected with \textnormal{Hilbert} space $\Hb$. The expression formula of ~$\Hb$ inner product in ~\textnormal{Hilbert} space is discussed in two cases.
	
	1. When $H>\frac12$, The covariance of $B_t^H$ can be written as
	\begin{align}
		\label{coviance}R_H(t,s)=\alpha_H\int^s_0\dif u\int^t_0|u-v|^{2H-2}\dif v,
	\end{align}
	where~$\alpha_H=H(2H-1)$. But for any $f,\,g\in\Hb$, we have
	\begin{align}\label{neijibiaochu 0}
		\langle f,g\rangle_\Hb=\alpha_H\int^T_0 g(s)\dif s\int^T_0 f(t)|s-t|^{2H-2}\dif t,
	\end{align}
	It should be noted that the elements in $\Hb$ are not necessarily ordinary functions.
	
	2. For any given $s\in[0,T]$, when $H<\frac12$, the defective integral of $|s-t|^{2H-2}$ on ~$[0,T]$ does not converge, so the covariance function of $B_t^H$ cannot be directly expressed in the form of ~\eqref{coviance}. At this time, the elements in space $\Hb$ are ordinary functions, but the formula \eqref{neijibiaochu 0} of inner product in \textnormal{Hilbert} space ~$\Hb$ is generally not true. But what is interesting is that the covariance of $B_t^H$ increment satisfies the formula:
    If let $0 \leq a < b \leq c < d \leq T$, then
	\begin{align}\label{bth zengliang}
		 \mathbb{E}[(B_b^H-B_a^H)(B_d^H-B_c^H)] =\alpha_H\int^b_a\dif u\int^d_c\,|u-v|^{2H-2}\dif v,
	\end{align}
	This leads to the conclusion that if $f,\,g\in\Hb$ supports are disjoint, the inner product formula \eqref{neijibiaochu 0} still holds, see \cite{Cheridito 03, Mishura 08}. By the way, we point out that Corollary ~\ref{coro-gu-3} in this paper can also lead to this known conclusion.
	
	In reference~\cite{Jolis}, a formula is given for limiting the inner product of~\textnormal{Hilbert} space~$\Hb$ to the bounded variation function $\V_{[0,T]}$. 
 If~$f,g\in\V_{[0,T]}$ then
		\begin{equation}\label{inner product_1}
		    	\langle f,g\rangle_\Hb=\int_{[0,T]^2}R_H(s,t)\nu_f(\dif s)\nu_g(\dif t)\\
			=-\int_{[0,T]^2}g(t)\frac{\partial R_H}{\partial t}(s,t)\dif t\,\nu_f(\dif s).
		\end{equation}
 refer to~\cite{HuNualartandZhou2019}~\cite{Jolis}~\cite{chenandli2021}.

In addition, with the help of Fourier transform, the following formula for the inner product of \textnormal{Hilbert} space~$\Hb$ is sometimes very useful:
\begin{align}\label{flybh njff}
\innp{f,\,g}_{\Hb}=\frac{\Gamma(2H+1)\sin(\pi H)}{2\pi}\int_{\Rnum} \mathcal{F}f (\xi)\overline{\mathcal{F}g (\xi)} \abs{\xi}^{1-2H}\,\dif \xi,
\end{align} here $f,g$ can be taken from some proper subspace of $\Hb$, see \cite{PipTaq 00} for details. 

Finally, with the help of kernel function
	\begin{align}
		K_H(t,s)=c_H\Big[\Big(\frac ts\Big)^{H-\frac12}(t-s)^{H-\frac12}-(H-\frac12)s^{\frac12-H}\int^t_su^{H-\frac32}(u-s)^{H-\frac12}\dif u\Big]
	\end{align}
and operator~$K^*_H$
	\begin{align*}
		 (K^*_H\phi)(t)=K_H(T,t)\phi(T)+\int^T_t\frac{\partial K_H}{\partial s}(s,t)[\phi(s)-\phi(t)]\dif s,
	\end{align*}
People transform the inner product of ~\textnormal{Hilbert} space ~$\Hb$ into the inner product of elements in~$L^2([0,T])$.
	\begin{align}
		 \langle\phi,\psi\rangle_{\Hb}=\langle K^*_H\phi(t),K^*_H\psi(t)\rangle_{L^2([0,T])}. \label{suanzi bianhuanfa}
	\end{align}
This inner product formula establishes the theoretical relationship between $\Hb$ and $L^2([0,T])$, but people usually do not directly use it to calculate the inner product. For details, see~\cite{David Nualart}.

    \subsection{A new formula for calculating inner product}
Let $f$, $g$ be monotone non decreasing functions on $\Rnum$, then the~\textnormal{Lebesgue-Stieljes} measure associated with bounded variation functions~$(f-g)$ on $\Rnum$ is defined as
\begin{align*}\bar{\nu}_{(f-g)}=\bar{\nu}_{f}-\bar{\nu}_{g}.\end{align*}
Here, $\bar{\nu}_{f}$ is the \textnormal{Lebesgue-Stieljes} positive measure on $\big(\Rnum,\mathcal{B}(\Rnum)\big)$ associated with monotone non decreasing function $f$ on $\Rnum$, and we emphasize that right continuity of $f$ is not required here. In fact, the value of function $f$ at the discontinuous point and its \textnormal{Lebesgue-Stieljes} measure $\bar{\nu}_{f}$ is independent (this point has been implicitly used many times in this paper). For details, see the Theorem 1.7.9 and exercise 1.7.12 in \cite{Tao}. From the uniqueness theorem of measures, it is easy to deduce the following well-known lemmas:
\begin{lemma}\label{measure}
If $F,G$ is are bounded variation functions on $\Rnum$, order~$\Psi=F+G$ then
	\begin{align}
		\bar{\nu}_{\Psi}=\bar{\nu}_{F}+\bar{\nu}_{G},\label{measure1}
	\end{align}
In particular, when $F,G\in\V_{[0,T]}$, order~$\Psi=F+G$ then
	\begin{align}
		\nu_{\Psi}=\nu_{F}+\nu_{G},\label{measure1}
	\end{align}
Here, $\nu_{\Psi}$ is the limit of the measure $\bar{\nu}_{\Psi^0}$ associated with the extension $\Psi^0$ of function $\Psi$ on $\Rnum$ on $\big([0,T],\, \mathcal{B}([0,T])\big)$, see notation \ref{jihao 001}.
\end{lemma}


The following Lemma~\ref{partial integral} on the Partial integral formula of measure is one of the main bases for proving Proposition \ref{key1}, which is taken from Lemma 3.1 of \cite{withdingzhen}. The key is to regard the value of the function at two endpoints in the general Partial integral formula as a measure about two Dirac points (or called Dirac $\delta$ generalized function). The two point measures are absorbed into the Lebesgue-Stieljes measure associated with the bounded variation function. This processing method first extends the bounded variation function, and then restricts the Lebesgue-Stieljes measure generated by the extended function back to the original support set of the bounded variation function. See \cite{Jolis, withdingzhen} for details, and the notation \ref{jihao 001} and formula \eqref{jieshi01} in this paper.
\begin{lemma}\label{partial integral}
    Let~$[a,b]$ be a compact interval with positive length, $\phi$:$[a,b]\rightarrow\R$ be continuous on~$[a,b]$ and differentiable on~$(a,b)$. If~$\phi'$ is absolutely integrable, then for any~$f\in\V_{[a,b]}$, we have
	\begin{align}
		-\int_{[a,b]}f(t)\phi'(t)\dif t=\int_{[a,b]}\phi(t) {\nu_f}(\dif t).\label{01}
	\end{align}
	where $\nu_f$ is a restriction on ~$\big({[a,b]},\mathcal{B}({[a,b]})\big)$ of the \textnormal{Lebesgue-Stieljes} measure on $\big(\Rnum,\mathcal{B}(\Rnum)\big)$ associated with
	\begin{equation*}
f^0(x)=\left\{
      \begin{array}{ll}
f(x), & \quad \text{if } x\in [a,b],\\
0, &\quad \text{other }.
 \end{array}
\right.
\end{equation*} 
\end{lemma}
\begin{remark} Lemma \ref{partial integral} is a rewriting of the Partial integral formula of continuous monotone increasing function (such as \cite{Tao} exercise 1.7.17). Its proof comes from Proposition 1.6.41 of \cite{Jolis} and \cite{Tao}. See Lemma 3.1 of \cite{withdingzhen} for details.
\end{remark}
From formulas \eqref{bth zengliang},\,\eqref{inner product_1}, 
and the above lemma, we have the following inference:
\begin{corollary}\label{coro-gu-3}
		Let $0 \leq a < b \leq c < d \leq T$, the bounded variation functions $f(s)$ and~$g(t)$ are supported on $[a,b]$ and $[c,d]$ respectively. If $H\in(0,\frac12)$, then
		\begin{align}\langle f,g\rangle_\Hb=\alpha_H\int^b_a f(s) \dif s \int^d_c g(t) (t-s)^{2H-2}\dif t \label{tl}.
		\end{align}
\end{corollary}
\begin{remark}
Since the set of bounded variation functions $\V_{[0,T]}$ is a dense subset of Hilbert space $\Hb$, the inner product formula \eqref{tl} based on the continuity of the inner product is still valid for any function supporting disjoint in Hilbert space $\Hb$.
\end{remark}
In the Partial integration formula of measure, namely Lemma~\ref{partial integral}, take $\phi\equiv 1$ to get:
\begin{corollary}\label{coro-3-chen1}
Let function $\varphi\in \V^{\otimes2}_{[0,T]}$, and $\frac{\partial_{a}}{\partial s},\,\frac{\partial_{a}\partial_{b}}{\partial s \partial t}$ as shown in mark \ref{jiaohao suanzi}. If functions $f,\,g$ are bounded \textnormal{Borel} measurable functions on set $[(a-1)\vee 0,\,(a+1)\wedge T]$ and set  $[(b-1)\vee 0,\,(b+1)\wedge T]$ respectively, then:
\begin{align*}
    \int^{(a+1)\wedge T}_{(a-1)\vee 0}\frac{\partial_{a}}{\partial s} \varphi(s,t) \dif s &=0,\\
   \int^{(b+1)\wedge T}_{(b-1)\vee 0} g(t)\dif t \int^{(a+1)\wedge T}_{(a-1)\vee 0} \frac{\partial_{a}\partial_{b}}{\partial s \partial t} \varphi(s,t) \dif s &=0,\\
   \int^{(a+1)\wedge T}_{(a-1)\vee 0}f(s) \dif s\int^{(b+1)\wedge T}_{(b-1)\vee 0}    \frac{\partial_{a}\partial_{b}}{\partial s \partial t} \varphi(s,t) \dif t &=0.
\end{align*}
\end{corollary}

In the rest of this section, we give the proof of Proposition~\ref{key1}.

	{\bf Proof of Proposition~\ref{key1}:}
        The method is to use measure decomposition. First, determine~$t\in[0,T] $, and divide~$s\in [0,T]$ into the following three intervals~$O_1:=[0, (t-1)\vee0), \, O_2:=[(t-1)\vee0,(t+1)\wedge T], O_3=((t+1)\wedge T,\, T]$;
	\begin{center}
		\begin{tikzpicture}
			\draw[->] (-2.5,-2) -- (2.9,-2);
			\draw[gray] (-2,-2) -- (-2,-2+0.1);
			\draw[gray] (2,-2) -- (2,-2+0.1);
			\draw[-] (0.678-0.6,-2) -- (0.678-0.6,-2+0.1);
			\draw[-] (0.678+0.4,-2) -- (0.678+0.4,-2+0.1);
			\node[above right] at (2+0.5,-2) {$s$};
			\node[below] at (-2,-2) {$0$};
			\node[below] at (2,-2) {$T$};
			\node[above] at (-1.011,-2) {$O_1$};
			\node[above] at (1.589,-2) {$O_2$};
			\node[above] at (0.678-0.1,-2) {$O_3$};
			\node[below] at (0.678+0.5,-2) {$t+1$};
			\node[below] at (0.678-0.7,-2) {$t-1$};
		\end{tikzpicture}  
	\end{center}
    Then the function $f(s)$ is decomposed into the restriction of the above three intervals, so that the measure $\nu_f$ is decomposed into the sum of the Lebesgue-Stieljes measures associated with the three. The specific steps are as follows:

		First, review formula \eqref{inner product_1}:
		\begin{align}
			\langle f,g\rangle_{\Hb}=-\int_0^Tg(t)\dif t\int^T_0\frac{\partial R_H}{\partial t}(s,t)\nu_{f}(\mathrm{d}s).\label{inner product_3}
		\end{align}
        For any given~$t\in[0,T]$, decompose function $f(s)\in\V_{[0,T]}$ as shown in the figure above:
		\begin{align*}
			\begin{split}
				f(s)&=f(s)\big(\1_{[0,(t-1)\vee0]}(s)+\1_{[(t-1)\vee0,(t+1)\wedge T]}(s)+\1_{((t+1)\wedge T,T]}(s)\big)\\
				:&=f_t^1(s)+\tilde{f_t}(s)+f_t^2(s).
			\end{split}
		\end{align*}
		According to Lemma~\ref{measure} 
		know the measure $\nu_f$ has the following decomposition:
		\begin{align}
			\nu_{f}=\nu_{f_t^1}+\nu_{\tilde{f_t}}+\nu_{f_t^2},\label{2disintegration}
		\end{align}
		Here, the four measures are \textnormal{Lebesgue-Stieljes} measures defined on $\big([0,T],\mathcal{B}([0,T]\big)$. Substitute formula~\eqref{2disintegration} into formula~\eqref{inner product_3} to get:
		\begin{align}
			\begin{split}
				\langle f,g\rangle_{\Hb}=&-\int_0^Tg(t)\dif t\int^T_0\frac{\partial R_H}{\partial t}(s,t)\nu_{f_t^1}(\mathrm{d}s)-\int_0^Tg(t)\dif t\int^T_0\frac{\partial R_H}{\partial t}(s,t)\nu_{\tilde{f}_t}(\mathrm{d}s)\\
				&-\int_0^Tg(t)\dif t\int^T_0\frac{\partial R_H}{\partial t}(s,t)\nu_{f_t^2}(\mathrm{d}s):=I_1+I_2+I_3\label{inner product}.
			\end{split}
		\end{align}
        Note that the support set of function $f_t^1(s)$ is $[0,(t-1)\vee0]$, then
		\begin{align*}
			\label{I_1}\int^T_0\frac{\partial R_H}{\partial t}(s,t)\nu_{f_t^1}(\mathrm{d}s)=
			\begin{cases}
				0 & \,\, t\in[0,1],\\
				\int^T_0 \1_{[0,t-1]}(s)\frac{\partial R_H}{\partial t}(s,t){{\nu}_{f_t^1}}(\mathrm{d}s)& \,\, t\in(1,T].
			\end{cases}
		\end{align*}
	Note that when $t\in(1,T]$, integral in the right end of the above equation $$\int^T_0 \1_{[0,t-1]}(s)\frac{\partial R_H}{\partial t}(s,t){{\nu}_{f_t^1}}(\mathrm{d}s)=\int^{t-1}_0\frac{\partial R_H}{\partial t}(s,t){{\nu}_{f_t^1}}(\mathrm{d}s),$$
    where, in fact, measure ${\nu}_{f_t^1}$ at the right end can be understood as only defined on $\big([0,t-1],\, \mathcal{B}([0,t-1])\big)$; however function $\frac{\partial^2R_H }{\partial s\partial t}(s,t)$, as a function of argument $s$, is absolutely integrable on $[0,t-1]$, so it is obtained from Lemma ~\ref{partial integral}:
		\begin{align*}
			\int^{t-1}_0\frac{\partial R_H}{\partial t}(s,t){{\nu}_{f_t^1}}(\mathrm{d}s)=-\int^{t-1}_0f(s)\frac{\partial^2R_H }{\partial s\partial t}(s,t)\dif s .
		\end{align*}
		Thus:
		\begin{equation}
			\begin{split}
				I_1
				=\alpha_H\int^T_{1} g(t) \dif t\Big(\int^{t-1}_0 f(s) (t-s)^{2H-2}\dif s\Big)\label{I1}.\end{split}
		\end{equation}
		Similarly, because the support set of function $f_t^2(s)$ is $[(t+1)\wedge T,\, T]$, we have
		\begin{align*}
			 \int^T_0\frac{\partial R_H}{\partial t}(s,t)\nu_{f_t^2}(\mathrm{d}s)=
			\begin{cases}
				\int_{t+1}^T\frac{\partial R_H}{\partial t}(s,t) {{\nu}_{f_t^2}}(\mathrm{d}s) & \,\, t\in[0,T-1),\\
				0 & \,\, t\in[T-1,T],
			\end{cases}
		\end{align*}
and
		\begin{align*}
			\int_{t+1}^T\frac{\partial R_H}{\partial t}(s,t) {{\nu}_{f_t^2}}(\mathrm{d}s)=-\int_{t+1}^Tf(s)\frac{\partial^2 R_H}{\partial s\partial t}(s,t)\dif s.
		\end{align*}
and
		\begin{align}\label{I3 new}
				I_3=\alpha_H\int^T_{1} f(s) \dif s\int^{s-1}_0 g(t) (s-t)^{2H-2}\dif t.  
		\end{align}
		Substitute formula \eqref{I1},~\eqref{I3 new} into formula \eqref{inner product} to get formula \eqref{inner product_2}. 
		Finally, the proof of formula \eqref{inner product_4} is the same.
		{\hfill\large{$\Box$}}
	

	\section{Proof of main Theorems}\label{sec three}
 Without losing generality,  {this section assumes the parameters $\theta=1$ in definitions \eqref{ftst defn} and \eqref{ht ts} of binary functions $f_T(t,s)$ and $h_T(t,s)$}. 
 
{\bf Proof of Proposition~\ref{jiaochaxiang}}:
    Firstly, let $t\in [0,T]$ be determined, and understand $f_T(t,\cdot)$ as a function of one variable on $s\in [0,T]$. Then notice that the binary function $h_T$ can be expressed as a function of one variable $\phi_T(t)=e^{t-T} \mathbb{1}_{[0,T]}(t)$ tensor about oneself, namely $h_T(t,s)=\phi_T(t)\phi_T(s)$. So according to the Fubini theorem, we have:
    \begin{align}\label{zhagnliang fenjie}
      \innp{f_T,\, h_T}_{{\Hb}^{\otimes 2}}=  \innp{\innp{f_T(t,\cdot),\, \phi_T}_{{\Hb} },\,\phi_T}_{{\Hb} }.
    \end{align}
    
    Secondly, calculate the inner product $\innp{f_T(t,\cdot),\, \phi_T}_{\Hb} $ when $t\in [0,T]$ is taken. According to the linear property of the inner product, we have:
    \begin{align}\label{zuihoufenjie}
        \innp{f_T(t,\cdot),\, \phi_T}_{\Hb} =\innp{f^1,\, h^1}_{\Hb}+\innp{f^1,\, h^2}_{\Hb}+\innp{f^2,\, h^1}_{\Hb}+\innp{f^2,\, h^2}_{\Hb},
    \end{align}here function
\begin{align*}
f^1(\cdot)&=f_T(t, \cdot) \mathbb{1}_{[0,t)}(\cdot),\quad f^2(\cdot)=f_T(t, \cdot) \mathbb{1}_{[t,T]}(\cdot)\\
    h^1(\cdot)&=\phi_T(\cdot) \mathbb{1}_{[0,t)}(\cdot),\qquad h^2(\cdot)=\phi_T(\cdot) \mathbb{1}_{[t,T]}(\cdot).
\end{align*} 
According to the support set and inner product calculation formula \eqref{inner product_1} or \eqref{neijibiaochu 0} of the above four functions, respectively we have:
\begin{align*}
      \innp{f^2,\, h^1}_{\Hb}&= \alpha_H e^{t-T}\int_t^T \dif u \int_0^t\dif v\, e^{-u +v } (u-v)^{2H-2} ,\\ 
   \innp{f^1,\, h^2}_{\Hb}&= \alpha_H e^{-t-T}\int_0^t\dif u\int_t^T \dif v\, e^{u +v } (v-u)^{2H-2},\\
   \innp{f^1,\, h^1}_{\Hb}&= -H e^{-t-T}\int_{[0,t]^2} e^{u+v}(1-\delta_t(u)) \big(v^{2H-1}-\abs{v-u}^{2H-1}\sgn{(v-u)} \big)\dif u\dif v,\\
   &=H e^{-T} \int_0^t (e^{v-t}+e^{t-v}) v^{2H-1}\dif v,\\
    \innp{f^2,\, h^2}_{\Hb}&= -H e^{t-T} \int_{[t,T]^2} e^{-u+v}(-1+\delta_t(u)-\delta_T(u)) \big(v^{2H-1}-\abs{v-u}^{2H-1}\sgn{(v-u)} \big)\dif u\dif v\\
    &=H e^{t-T} \Big[\int_0^{T-t}\dif y\int_0^y (e^{-x}-e^{x})x^{2H-1}\dif x + \int_0^{T-t} (e^{-x}+e^{x})x^{2H-1}\dif x\Big].
\end{align*}

Then we enlarge the inner product of the above items and  {$\phi_T$}: record $i,j=1,2$,
\begin{align}\label{fangda psiT}
     \abs{\innp{\innp{ {f^i},\, h^j}_{\Hb},\,\phi_T}_{\Hb}}\le  \int_0^T \dif t \abs{\innp{ {f^i},\, h^j}_{\Hb}}\abs{\int_0^T  e^{s-T}(1-\delta_T(s))\frac{\partial R(t,\,s)   }{\partial t} \dif s}. 
\end{align} 
We then assert that there is a constant $C_H$ independent of $T$, so that for any given $t\in [0,T]$, and we have inequality 
\begin{align}
 \abs{\int_0^T e^{s-T}(1-\delta_T(s))\frac{\partial R(t,\,s)   }{\partial t}  \dif s}&\le C_H \times \Big[ e^{-T} t^{2H-1}+ e^{t-T}+ (T-t)^{2H-1} \mathbb{1}_{(T-1,T]}(t) \notag\\
 &+ (T-t)^{2H-2} \mathbb{1}_{[0,T-1]}(t)\Big]\label{chudengbudengshi2}
\end{align} holds. In fact,
\begin{align*}
 &\abs{\int_0^T e^{s-T}(1-\delta_T(s))\frac{\partial R(t,\,s)   }{\partial t}  \dif s}\\
 &=H \abs{ t^{2H-1}\big(\int_0^T e^{s-T} \dif s -1) +e^{t-T}\Big[-\int_0^t e^{s-t}(t-s)^{2H-1}\dif s  + \int_t^T e^{s-t}(s-t)^{2H-1} \dif s\Big]-(T-t)^{2H-1} }\\
 &\le C_H\times \Bigg[ t^{2H-1}e^{-T}+\abs{e^{t-T}\Big[{-\int_0^t e^{-u}u^{2H-1}\dif u  + 2\int_0^{1 } e^{u} u^{2H-1} \dif u}\Big]}+(T-t)^{2H-1} \mathbb{1}_{(T-1,T]}(t)\\
 &+  \mathbb{1}_{[0,T-1]}(t)\abs{ e^{t-T}\int_1^{T-t} {e^u}u^{2H-1}\dif u -(T-t)^{2H-1}}\\
 &\le C_H\times  \Big[ e^{-T} t^{2H-1}+ e^{t-T}  + (T-t)^{2H-1} \mathbb{1}_{(T-1,T]}(t)+ (T-t)^{2H-2} \mathbb{1}_{[0,T-1]}(t) 
\Big],
\end{align*}
where, the previous inequality can be seen in Lemma 2.2 of \cite{Cheridito 03}.

Finally, it is easy to see that there is a constant $C_H$ independent of $T$, which makes
\begin{align}\label{zuihou 1}
    \int_{T-1}^T \abs{\innp{ {f^i},\, h^j}_{\Hb}} (T-t)^{2H-1}\dif t\le C_H,
\end{align} and 
\begin{align*}
     \abs{\innp{ {f^i},\, h^j}_{\Hb}}\le C_H
\end{align*} hold for any $i,j=1,2$. From this inequality we have
\begin{align}\label{ziuhou 2}
    \int_{0}^T \abs{\innp{ {f^i},\, h^j}_{\Hb}} \Big[ e^{-T} t^{2H-1}+ e^{t-T} + (T-t)^{2H-2} \mathbb{1}_{[0,T-1]}(t) \Big] \dif t\le C_H.
\end{align}
Combining inequalities \eqref{fangda psiT}, \eqref{zuihou 1} and \eqref{ziuhou 2}: There is a constant $C_H$ independent of $T$, so that
\begin{align*}
      \abs{\innp{\innp{ {f^i},\, h^j}_{\Hb},\,\phi_T}_{\Hb}}\le C_H.
\end{align*}
According to the identities \eqref{zhagnliang fenjie} and \eqref{zuihoufenjie}, we can get the inequality \eqref{touyige budsh}.
{\hfill\large{$\Box$}}

{\bf Proof of Theorem~\ref{qq1}}:
    Recall that in Remark~1.1.2, we will draw a conclusion that is stronger than the formula \eqref{main} required by the theorem. That is, the square of the norm of the binary function $f_T(t,s)$ is taken as the function of $T$ and the asymptote \eqref{jianjinxian mubiao} when $T\to \infty$. Equation ~\eqref{jianjinxian mubiao} is proved in several steps as follows.
	
	\textbf{Step~1}. According to Proposition \ref{key2}, we get the decomposition formula of~$\|f_T\|_{\Hb^{\otimes2}}^2$.
	\begin{align*}
\begin{split}
			\|f_T\|_{{\Hb}^{\otimes2}}^2=&\alpha_H^2\sum_{i,j=1}^2\int_{\kappa_i\times \kappa_j}e^{-|s_1-t_1|}e^{-|s_2-t_2|}|s_1-s_2|^{2H-2}|t_1-t_2|^{2H-2}\dif \vec{s}\dif \vec{t}\\
		&-2\alpha_H\sum_{i=1}^2\int_{\kappa_3\times \kappa_i}e^{-|s_1-t_1|}\frac{\partial_{s_1}}{\partial s_2}e^{-|s_2-t_2|}\frac{\partial R_H}{\partial s_1}(s_1,s_2)\abs{t_1-t_2}^{2H-2}\dif \vec{s}\dif \vec{t}\\
		&+\int_{\kappa_3\times \kappa_3}e^{-|s_1-t_1|}\frac{\partial R_H}{\partial s_1}(s_1,s_2)\frac{\partial R_H}{\partial t_1}(t_1,t_2)\frac{\partial_{s_1}\partial_{t_1}}{\partial s_2\partial t_2}e^{-|s_2-t_2|} \dif \vec{s}\dif \vec{t}\\
		:=&\alpha_H^2 \sum_{i,j=1}^2M_{ij}(T) -2 \alpha_H\sum_{i=1}^2 M_{3i}(T) +M_{33}(T). 
\end{split}
	\end{align*}
Making the change of variables~$x=T-s_1,\,y=T-t_1,\, u=T-s_2,\, v=T-t_2$, we have:
	\begin{align*}
	M_{11}(T)=M_{22}(T)\,\,\,\text{\, and \,}\,\,\,M_{12}(T)=M_{21}(T).
	\end{align*}
So
\begin{align}
    	\|f_T\|_{{\Hb}^{\otimes2}}^2=M_{33}(T)+ 2\Big(\alpha_H^2\big(M_{11}(T)+M_{12}(T)\big)-\alpha_H\big(M_{31}(T)+M_{32}(T)\big)\Big).\label{fT norm2 fenjieshizi chufadian}
\end{align}	

\textbf{Step~2}. Solve the asymptote of function $M_{11}(T)+M_{12}(T)$ when $T\to\infty$. First,
	\begin{small}
	\begin{align}
M_{11}(T)&=\int^T_1\dif s_1\int^{T}_1e^{-\abs{s_1-t_1}}\dif      t_1\int^{s_1-1}_{0}(s_1-s_2)^{2H-2}\dif s_2\int^{t_1-1}_{0} (t_1-t_2)^{2H-2}e^{-|t_2-s_2|}\dif t_2\notag\\
	&=2\int^T_1e^{-s_1}\dif s_1\int^{s_1}_1e^{t_1}\dif t_1\int^{s_1-1}_{0}(s_1-s_2)^{2H-2}\dif s_2\int^{t_1-1}_{0} (t_1-t_2)^{2H-2}e^{-|t_2-s_2|}\dif t_2.\label{M11 biaodashi}\\
M_{12}(T)&=\int^T_1\dif s_1\int^{T-1}_0\dif t_1\int^{s_1-1}_0 (s_1-s_2)^{2H-2}\dif s_2\int^T_{t_1+1}(t_2-t_1)^{2H-2}e^{-|t_1-s_1|-|t_2-s_2|}\dif t_2. \label{M12T biaodashi}	
\end{align} 
\end{small}
Formula \eqref{M11 biaodashi} is the symmetry of integral with respect to two variables $s_1,t_1$. According to Lemma \ref{M11} and Lemma \ref{yyijh}, we get the function $M_{11}(T)+ M_{12}(T)$, and the asymptote when $T\to\infty$ is:
	\begin{align}
	T&\times \Bigg[(4H-1)\big(\int_1^{\infty} e^{-u} u^{2H-2}\dif u\big)^2+2\int_1^{\infty} (e^{1-u}+e^{-1-u})u^{2H-2}\dif u\notag\\
	&+2(4H-1)\int_1^{\infty} e^{-u} u^{2H-2} \dif u\int_1^u  e^{v} v^{2H-2}\dif v\Bigg] +C_H. \label{M11+M12 jianjinxian}
	\end{align}

\textbf{Step~3}. Solve the asymptote of function $M_{31}(T)+M_{32}(T)$ when $T\to\infty$. First,
\begin{align}
M_{31}(T)&=\int_{\kappa_3\times \kappa_1}e^{-|s_1-t_1|}\frac{\partial_{s_1}}{\partial s_2}e^{-|s_2-t_2|}\frac{\partial R_H}{\partial s_1}(s_1,s_2)(t_1-t_2)^{2H-2}\dif \vec{s}\dif \vec{t}\notag\\
&=H \int^{T}_{0}\dif s_1\int^T_1 e^{-|s_1-t_1|}\dif t_1\int^{t_1-1}_0 (t_1-t_2)^{2H-2}\dif t_2\int^{(s_1+1)\wedge T}_{(s_1-1)\vee 0}\frac{\partial_{s_1}}{\partial s_2}e^{-|s_2-t_2|}\notag\\
&\times \big(s_1^{2H-1}-\mathrm{sgn}(s_1-s_2)|s_1-s_2|^{2H-1}\big) \dif s_2. \label{m31 biaoshi}
	\end{align}
From Inference \ref{coro-3-chen1}, we have:
		\begin{align*}
\int^{T}_{0}s_1^{2H-1} \dif s_1\int^T_1e^{-|s_1-t_1|}\dif t_1\int^{t_1-1}_0 (t_1-t_2)^{2H-2}\dif t_2\int^{(s_1+1)\wedge T}_{(s_1-1)\vee 0}\frac{\partial_{s_1}}{\partial s_2}e^{-|s_2-t_2|} \dif s_2=0.
	\end{align*}
Put it in \eqref{m31 biaoshi}, we have
\begin{align}  
M_{31}(T)&=-H \int^{T}_{0}\dif s_1\int^T_1 e^{-|s_1-t_1|}\dif t_1\int^{t_1-1}_0 (t_1-t_2)^{2H-2}\dif t_2\int^{(s_1+1)\wedge T}_{(s_1-1)\vee 0}\frac{\partial_{s_1}}{\partial s_2}e^{-|s_2-t_2|}\notag\\
&\times \mathrm{sgn}(s_1-s_2)|s_1-s_2|^{2H-1} \dif s_2:=H\times [N(T)-\tilde{N}(T)]
,
	\end{align} here
\begin{align}
N(T)&=\int^T_0\dif s_1\int^T_1e^{-|t_1-s_1|}\dif t_1\int^{t_1-1}_0 (t_1-t_2)^{2H-2} \dif t_2\int^{(s_1+1)\wedge T}_{(s_1-1)\vee 0}\sgn(s_1-s_2)|s_1-s_2|^{2H-1}\notag\\
&\times \sgn(s_2-t_2)e^{-|t_2-s_2|}\dif s_2 ;\label{nT expression}\\
\tilde{N}(T)&= \int^T_0\dif s_1\int^T_1 e^{-|t_1-s_1|}\dif t_1\int^{t_1-1}_0 (t_1-t_2)^{2H-2} \dif t_2\int^{(s_1+1)\wedge T}_{(s_1-1)\vee 0} e^{-|t_2-s_2|}\notag\\
&\times \sgn(s_1-s_2)|s_1-s_2|^{2H-1} \big( \delta_{(s_1-1)\vee 0}(s_2)-\delta_{(s_1+1)\wedge T}(s_2)\big)\dif s_2. \label{bar nT expression}
\end{align}
Similarly, we have: 
\begin{align}  
M_{32}(T)&=-H \int^{T}_{0}\dif s_1\int^T_1 \dif t_2\int^{t_2-1}_0 e^{-|s_1-t_1|}  (t_2-t_1)^{2H-2} \dif t_1\int^{(s_1+1)\wedge T}_{(s_1-1)\vee 0}\frac{\partial_{s_1}}{\partial s_2}e^{-|s_2-t_2|}\notag\\
&\times \mathrm{sgn}(s_1-s_2)|s_1-s_2|^{2H-1} \dif s_2:=H\times [U(T)-\tilde{U}(T)],
	\end{align}	here
\begin{align}
U(T)&=\int^T_0\dif s_1\int^T_1\dif t_2\int^{t_2-1}_0 e^{-|t_1-s_1|}(t_2-t_1)^{2H-2} \dif t_1\int^{(s_1+1)\wedge T}_{(s_1-1)\vee 0} \sgn(s_2-t_2) e^{-|t_2-s_2|} \notag\\
&\times  \sgn(s_1-s_2)|s_1-s_2|^{2H-1}\dif s_2 ;\label{UT expression}\\
\tilde{U}(T)&= \int^T_0\dif s_1\int^T_1 \dif t_2\int^{t_2-1}_0 e^{-|t_1-s_1|} (t_2-t_1)^{2H-2} \dif t_1\int^{(s_1+1)\wedge T}_{(s_1-1)\vee 0} e^{-|t_2-s_2|}\notag\\
&\times \sgn(s_1-s_2)|s_1-s_2|^{2H-1} \big( \delta_{(s_1-1)\vee 0}(s_2)-\delta_{(s_1+1)\wedge T}(s_2)\big)\dif s_2. \label{bar UT expression}
\end{align}
According to Lemma \ref{2.2}, Lemma \ref{m31.2.3} and Lemma \ref{UT BarUT jianjx yinli}, we get that the asymptote of $M_{31}(T)+M_{32}(T)$ is:
	\begin{align}
2 H T&\times \Bigg[\int_1^{\infty} e^{-u}u^{2H-1}\dif u \Big[ -2H(e^{-1}+e) +(4H-1) \int_0^1 (e^{x}+e^{-x})x^{2H-1}\dif x\Big]\notag\\ &+e^{-1}\int_0^1 (e^{x}-e^{-x})x^{2H-1}\dif x -(1+e^{-2}) \Bigg].
\label{M31+M32 jianjinxian}	\end{align}

	\textbf{Step~4}. Solve the asymptote of function $M_{33}(T)$.
Similar to the method for dealing with items $M_{31}(T)$ and $M_{32}(T)$ in step~3, we expand $\frac{\partial R_H}{\partial s_1}(s_1,s_2)$ and $\frac{\partial R_H}{\partial t_1}(t_1,t_2)$ in turn, and use Inference \ref{coro-3-chen1} twice consecutively to obtain:
	\begin{align}
	   M_{33}(T)&= \int_{\kappa_3\times \kappa_3}e^{-|s_1-t_1|}\frac{\partial R_H}{\partial s_1}(s_1,s_2)\frac{\partial R_H}{\partial t_1}(t_1,t_2)\frac{\partial_{s_1}\partial_{t_1}}{\partial s_2\partial t_2}e^{-|s_2-t_2|} \dif \vec{s}\dif \vec{t}\notag\\
	   &=H\int_0^T \dif s_1\int_0^T e^{-\abs{s_1-t_1}}\dif t_1\int^{(t_1+1)\wedge T}_{(t_1-1)\vee0} \frac{\partial R_H}{\partial t_1}(t_1,t_2) \dif t_2 \notag\\
	   &\times\int^{(s_1+1)\wedge T}_{(s_1-1)\vee0} \big(s_1^{2H-1}-\abs{s_1-s_2}^{2H-1}\sgn{(s_1-s_2)}\big)
	   \frac{\partial_{s_1}\partial_{t_1}}{\partial s_2\partial t_2}e^{-|s_2-t_2|}\dif s_2\notag\\
	   &=-H^2\int_0^T \dif s_1\int_0^T e^{-\abs{s_1-t_1}}\dif t_1\int^{(s_1+1)\wedge T}_{(s_1-1)\vee0}  \abs{s_1-s_2}^{2H-1}\sgn{(s_1-s_2)} 
	   \dif s_2\notag\\
	   &\times\int^{(t_1+1)\wedge T}_{(t_1-1)\vee0}  \big(t_1^{2H-1}-\abs{t_1-t_2}^{2H-1}\sgn{(t_1-t_2)}\big)\frac{\partial_{s_1}\partial_{t_1}}{\partial s_2\partial t_2}e^{-|s_2-t_2|} \dif t_2 \notag\\
	   &=H^2\int_0^T \dif s_1\int_0^T e^{-\abs{s_1-t_1}}\dif t_1\int^{(s_1+1)\wedge T}_{(s_1-1)\vee0}  \abs{s_1-s_2}^{2H-1}\sgn{(s_1-s_2)} 
	   \dif s_2\notag\\
	   &\times\int^{(t_1+1)\wedge T}_{(t_1-1)\vee0}   \abs{t_1-t_2}^{2H-1}\sgn{(t_1-t_2)} \frac{\partial_{s_1}\partial_{t_1}}{\partial s_2\partial t_2}e^{-|s_2-t_2|} \dif t_2. \label{mss33biaodashi}
	\end{align}
	Note that the bivariate joint ``density function'' in the above equation can be expressed as:
	\begin{align*}
	   \frac{\partial_{s_1}\partial_{t_1}}{\partial s_2\partial t_2}e^{-|s_2-t_2|} &=e^{-\abs{s_2-t_2}}\times \Big(-1-\mathrm{sgn}(t_2-s_2)[\delta_{(s_1-1)\vee0}(s_2)-\delta_{(s_1+1)\wedge T}(s_2)]\\
			&-\mathrm{sgn}(s_2-t_2)[\delta_{(t_1-1)\vee1}(t_2)-\delta_{(t_1+1)\wedge T}(t_2)]\\
			&+\big[(\delta_{(s_1-1)\vee0}-\delta_{(s_1+1)\wedge T})(s_2)(\delta_{(t_1-1)\vee0}-\delta_{(t_1+1)\wedge T})(t_2)\big]\Big) 
	\end{align*}
Substitute the above joint density function into formula \eqref{mss33biaodashi} to obtain:
\begin{align}\label{mss 33fenjie}
   M_{33}(T)&=H^2\times[{ -} L(T) +2 P(T)+Q(T)] ,
\end{align} where
\begin{align}
  	L(T)&=\int_{[0,T]^2}e^{-|t_1-s_1|} \dif s_1 \dif t_1\int^{(s_1+1)\wedge T}_{(s_1-1)\vee0}\dif s_2\int^{(t_1+1)\wedge T}_{(t_1-1)\vee0}\sgn(s_1-s_2)|s_1-s_2|^{2H-1}\notag\\
				&\times\sgn(t_1-t_2)|t_1-t_2|^{2H-1}e^{-|t_2-s_2|}\dif t_2,\label{LT biaodashi}\\
P(T)&=	\int_{[0,T]^2} e^{-|t_1-s_1|}\dif s_1 \dif t_1\int^{(s_1+1)\wedge T}_{(s_1-1)\vee0}\dif s_2\int^{(t_1+1)\wedge T}_{(t_1-1)\vee0}e^{-|t_2-s_2|}\sgn(s_1-s_2)|s_1-s_2|^{2H-1}\notag\\
			&\times\sgn(t_1-t_2)|t_1-t_2|^{2H-1} \sgn(s_2-t_2) [\delta_{(s_1-1)\vee0}(s_2)-\delta_{(s_1+1)\wedge T}(s_2)]\,\dif t_2,\label{PT biaodashi}\\
Q(T)&= \int_{[0,T]^2} e^{-|t_1-s_1|}\dif s_1 \dif t_1\int^{(s_1+1)\wedge T}_{(s_1-1)\vee0}\dif s_2\int^{(t_1+1)\wedge T}_{(t_1-1)\vee0}e^{-|t_2-s_2|} \sgn(s_1-s_2)|s_1-s_2|^{2H-1}\notag\\
				&\times\sgn(t_1-t_2)|t_1-t_2|^{2H-1}\big[(\delta_{(s_1-1)\vee0}-\delta_{(s_1+1)\wedge T})(s_2)(\delta_{(t_1-1)\vee0}-\delta_{(t_1+1)\wedge T})(t_2)\big]\dif t_2.\notag
\end{align}
For term $Q(T)$, first integrate Dirac function, then convert it into four double integrals, when $T\to\infty$, we can directly calculate that its asymptote is $${(6e^{-2}+2)}T+C_H.$$
Lemma \ref{yyuhjkl} and Lemma \ref{yyuhjkl02} respectively give the asymptotes of the terms $L(T)$ and $P(T)$ when $T\to\infty$, and combine the three asymptotes. According to formula \eqref{mss 33fenjie}, the asymptote of $M_{33}(T)$ is:
\begin{align}
 2H^2T&\times\Bigg[{-2(4H+1)\int_0^1e^{-u}u^{2H-1}\dif u\int_0^u e^v v^{2H-1}\dif v+(4H+1)\big(\int_0^1e^{-u}u^{2H-1}\dif u \big)^2 }\notag\\
 &{+4H\int_0^{1} \big( e^{-1-u}-e^{-1+u}\big)u^{2H-1}\dif u + e^{-2}+3}\Bigg]+C_H. \label{M33 jianjinxian}
\end{align}

Finally, the above three steps give the asymptotes of $M_{11}(T)+M_{12}(T)$, $M_{31}(T)+M_{32}(T)$, and $M_{33}(T)$ as functions of $T$ when $T\to \infty$, and obtain \eqref{M11+M12 jianjinxian}, \eqref{M31+M32 jianjinxian} and \eqref{M33 jianjinxian} respectively. Then it is known from the decomposition formula \eqref{fT norm2 fenjieshizi chufadian} that the norm of the binary function $f_T(t,s)$ is taken as the function of $T$, and the asymptote exists when $T\to \infty$. Finally, from the uniqueness of \eqref{zhou jielun} and function limit, it is concluded that \eqref{jianjinxian mubiao} holds.{\hfill\large{$\Box$}}
\begin{remark}
A by-product of the proof of Theorem \ref{qq1} is the following seemingly tedious analytical identities:
\begin{align*}
    2(H\Gamma(2H))^2\Big[4H-1 + \frac{2\Gamma(2-4H)\Gamma(4H)}{\Gamma(2H)\Gamma(1-2H)}\Big] =A_3+2\alpha_H(\alpha_H A_1-   A_2),
\end{align*} 
Here $A_1,A_2,A_3$ are the slope values of asymptotes \eqref{M11+M12 jianjinxian}, \eqref{M31+M32 jianjinxian} and \eqref{M33 jianjinxian} respectively. It should be emphasized that the true meaning of Theorem \ref{qq1} lies in ``The norm of the bivariate function $f_T(t,s)$ is taken as a function of $T$. When $T\to \infty$, the existence of asymptote.'' As for the specific value of the slope of the asymptote, it is not so important. Therefore, in order to save space, this paper will not verify this analytic identity.
\end{remark}

{\bf Proof of Theorem~\ref{zhuyaodl
2}}: First, we have to prove the \textnormal{Berry--Ess\'{e}en} type inequality \eqref{zuixiaoerch B_E } of the least squares estimator.
Therefore, From the proof of \cite{chenkuangandli2019} Theorem 1.1, we know that
\begin{align}
     & \sup_{z\in \Rnum}\abs{P(\sqrt{\frac{T}{\theta \sigma^2_H}} (\hat{\theta}_T-\theta )\le z)-P(Z\le z)}\notag\\
     &\le  C_{\theta, H}\times \max\set{\frac{1}{\sqrt{ T} },\, \abs{\frac{1}{T}\|f_T\|^2_{{\Hb}^{\otimes2}}-2(H\Gamma(2H))^2\sigma_H^2} }. \label{BE jie guji 1}
\end{align} Therefore, according to Theorem \ref{qq1}, we can know that inequality \eqref{zuixiaoerch B_E } is true.

Then we prove \textnormal{Berry--Ess\'{e}en} inequality of moment estimator \eqref{juguji de B_E }. According to the proof of Theorem 1.1 in \cite{Chen guli 22},
\begin{align}
     & \sup_{z\in \Rnum}\abs{P(\sqrt{\frac{T}{\theta \sigma^2_H}} (\tilde{\theta}_T-\theta )\le z)-P(Z\le z)}\notag\\
     &\le  C_{\theta, H}\times \max\set{\frac{1}{\sqrt{ T} },\, \abs{\frac{1}{T}\|f_T\|^2_{{\Hb}^{\otimes2}}-2(H\Gamma(2H))^2\sigma_H^2},\,\sqrt{\frac{\abs{\innp{f_T,\,h_T}_{{\Hb}^{\otimes2}}}}{T}  } }.\label{BE jie guji 2}
\end{align} Therefore, the inequality \eqref{juguji de B_E } is established by Theorem \ref{qq1} and Proposition. \ref{jiaochaxiang}
{\hfill\large{$\Box$}}

\section {Appendix: Asymptotic analysis}\label{jianjingfenxi}
	The Lemma \ref{e^uu^a} is a trivial inequality which is used several times in this section. 
\begin{lemma}\label{e^uu^a}
Let $\alpha<0$. There is a positive number $C_{\alpha}$ which depends only on $\alpha$ such that for any $x>1$, we have 
	\begin{align}
		\int^x_1e^uu^{\alpha}\dif u<C_{\alpha} \times e^x\label{e^uu^a1}.
	\end{align}
\end{lemma}

\begin{lemma}\label{M11}
Let \eqref{M11 biaodashi} be the expression of the quadruple integral $M_{11}(T)$. When $T\to\infty$, The asymptote of $M_{11}(T)$ is given by 
\begin{align}
2T\times \Big[(4H-1)\int^\infty_1e^{-v}v^{2H-2}\dif v\int^v_1e^u u^{2H-2}\dif u+\int^{\infty}_1e^{1-u}u^{2H-2}\dif u \big] +C_H.\label{M!!T jianjinxian}
	\end{align}
\end{lemma}
\begin{proof}
We take the integral variable $(s_1,\,t_1)$ of the quadruple integral $\frac{1}{2} M_{11}(T)$. We first decompose the region $[0,s_1-1]\times[0,t_1-1]$ of the integral variable $(s_2,t_2)$ into $\set{0 \le t_2\le t_1-1,\, t_2\le s_2\le s_1-1]}\cup\set{0\le s_2\le t_2\le t_1-1}$.  The integral $\frac{1}{2} M_{11}(T)$ restricted to the corresponding subregion is called $J_1(T),\,J_2(T)$ respectively, where

	\begin{align*}
	J_1(T)&=\int^T_1\dif s_1\int^{s_1}_1\dif t_1\int^{t_1-1}_0\dif t_2\int^{s_1-1}_{t_2}\dif s_2  (s_1-s_2)^{2H-2}(t_1-t_2)^{2H-2}e^{t_1-s_1-|t_2-s_2|},\\
J_2(T)&=\int^T_1\dif s_1\int^{s_1}_1\dif t_1 \int^{t_1-1}_0\dif t_2\int^{t_2}_0\dif s_2 (s_1-s_2)^{2H-2}(t_1-t_2)^{2H-2}e^{t_1-s_1-|t_2-s_2|}.
	\end{align*}
	
Then we try to obtain the asymptote of the quadruple integral $J_1(T)$. Making the change of variables  $u=s_1-s_2,\, v=t_1-t_2,\, x=s_1-t_1+v$. By the symmetry, we have
	\begin{align*}
		J_1(T)=&\int^T_1\dif s_1\int^{s_1}_1e^{-2x}\dif x\int^{x}_1e^{v}v^{2H-2}\dif v\int^{x}_1e^{u}u^{2H-2}\dif u\\
		=&2\int^T_1\dif s_1\int^{s_1}_1e^{-2x}\dif x\int^{x}_1e^{v}v^{2H-2}\dif v\int^{v}_1e^{u}u^{2H-2}\dif u.
	\end{align*}
Lemma \ref{e^uu^a} and Partial integration formulate indicate that when $T\to\infty$, the asymptote of $J_1(T)$ is
	\begin{align}
	T\times \int^\infty_1e^{-v}v^{2H-2}\dif v\int^v_1e^u u^{2H-2}\dif u+C_H.	 \label{J1T jianjx}
	\end{align}
Then we take the asymptote of $J_2(T)$. Making the change of variables $u=s_1-s_2,\,v=t_1-t_2,\, x=s_1-t_1+v $, we have
		\begin{align*}
			J_2(T)=\int^T_1\dif s_1\int^{s_1}_1e^{-u}u^{2H-2}\dif u\int^{u}_1\dif x\int^{x}_{1}e^{v}v^{2H-2}\dif v
		\end{align*}
Lemma $\ref{e^uu^a}$ and Partial integration formulate imply that when $T\to\infty$, the  asymptote of $J_2(T)$ is 
	\begin{align}
	&T\times \int^{\infty}_1e^{-u}u^{2H-2}\dif u\int^{u}_1\dif x\int^{x}_1e^{v}v^{2H-2}\dif v+C_H\notag\\
	&=T\times \int^{\infty}_1e^{-u}u^{2H-2}\dif u\int^{u}_1 e^{v}v^{2H-2}(u-v)\dif v+C_H.	\label{J2T jianjx}	 
	\end{align}
Finally, the  asymptote of $\frac12M_{11}(T)$ can be obtained by combining \eqref{J1T jianjx} and \eqref{J2T jianjx}, and we obtain the  asymptote \eqref{M!!T jianjinxian} of $M_{11}(T)$.

			\end{proof}

\begin{lemma}\label{yyijh}
Let \eqref{M12T biaodashi} be the expression of the quadruple integral $M_{12}(T)$, then when $T\to\infty$, the asymptote of $M_{12}(T)$ is 
				\begin{align}
 T \Bigg[ (4H-1)\big(\int_1^{\infty} e^{-u} u^{2H-2}\dif u \big)^2 + 2\int_1^{\infty} e^{-1-u} u^{2H-2}\dif u  \Bigg]+C_H. \label{BarJ jianjinxian}
				\end{align}
\end{lemma}
\begin{proof}
We first decompose the integral region $[1,T]\times[0,T-1]$ of the integral variable $(s_1,\,t_1)$ of the quadruple integral $M_{12}(T)$ into $\set{0\le s_1-1\le t_1\le T-1}\cup\set{0\le t_1\le s_1-1\le T-1}$. The integral $M_{12}(T)$ restricted to corresponding subregion is $\bar{J}_1(T)$ and $\bar{J}_2(T)$ respectively, where
	\begin{align*}
				\bar{J}_1(T)=& \int^T_1\dif s_1\int_{s_1-1}^{T-1}\dif t_1 \int_{0}^{s_1-1}\dif s_2\int^T_{t_1+1} (s_1-s_2)^{2H-2}(t_2-t_1)^{2H-2}e^{-|t_1-s_1|-|t_2-s_2|}\dif t_2,\\
				\bar{J}_2(T)=& \int_1^T\dif s_1\int_0^{s_1-1}\dif t_1 \int_{0}^{s_1-1}\dif s_2\int^T_{t_1+1}(s_1-s_2)^{2H-2}(t_2-t_1)^{2H-2}e^{-|t_1-s_1|-|t_2-s_2|}\dif t_2.
\end{align*}
Then we try to obtain the asymptote of $\bar{J}_1(T)$ when $T\to\infty$. Making the change of variables $u=t_2-t_1,\,v=u+s_1-1-s_2$ and $x=t_2-s_2$	, we have		
				\begin{align*}
					\bar{J}_1(T)=\int^T_1\dif t_2\int^{t_2}_1e^{-x}\dif x\int_1^{x}e^{-|x-v-1|}\dif v\int^v_1  (v-u+1)^{2H-2}u^{2H-2}\dif u
				\end{align*}
By the Partial integration formulate and Fubini Theorem, when $T\to\infty$, the asymptote is as follow 
				\begin{align}
 &T\times \int^{\infty}_1e^{-x}\dif x\int_1^{x}e^{-|x-v-1|}\dif v\int^v_1 (v-u+1)^{2H-2}u^{2H-2}\dif u +C_H\notag\\
 &=T\times \Bigg[\int^{\infty}_1 u^{2H-2}\dif u \int_u^{\infty} e^v(v-u+1)^{2H-2}  \dif v \int^{\infty}_{1+v}  e^{1-2x}\dif x\notag\\
 &+\int^{\infty}_1 u^{2H-2}\dif u \int_u^{\infty} e^{-v-1}(v-u+1)^{2H-2}  \dif v \int_{v}^{1+v} \dif x \Bigg]+ C_H\notag\\
 &=\frac{3}{2}T\times \Big(\int_1^{\infty} e^{-u} u^{2H-2}\dif u \Big)^2+ C_H. \label{barJ1T  jianjx}
				\end{align}
We next obtain the asymptote of the integral 	$\bar{J}_2(T)$ when $T\to\infty$. Fix the integral variable $(s_1,t_2)$ of $\bar{J}_2(T)$. We decompose the integral region $[0,s_1-1]^2 $ of integral variable 	$(t_1,s_2)$	into $\set{0\le s_2\le t_1\le s_1-1}\cup\set{0\le t_1\le s_2\le s_1-1}$. The integral $\bar{J}_2(T)$ restricted to corresponding subregion is $\bar{J}_{21}(T)$ and $\bar{J}_{22}(T)$ respectively, where	
	\begin{align*}
			\bar{J}_{21}(T)=&\int^T_1\dif s_1\int_0^{s_1-1}\dif t_1\int_0^{t_1}\dif s_2 \int^T_{t_1+1}(s_1-s_2)^{2H-2}(t_2-t_1)^{2H-2}e^{-|t_1-s_1|-|t_2-s_2|}\dif t_2,\\
			\bar{J}_{22}(T)=&\int^T_1\dif s_1 \int_0^{s_1-1}\dif s_2\int^{s_2}_0\dif t_1\int^T_{t_1+1} (s_1-s_2)^{2H-2}(t_2-t_1)^{2H-2}e^{-|t_1-s_1|-|t_2-s_2|}\dif t_2.
	\end{align*}
Fix the integral variable $(s_1,t_1, s_2)$ of $\bar{J}_{21}(T)$ again. We decompose the integral region $[t_1+1,\,T]$ of integral variable $t_2$	into $[t_1+1,\,s_1]\cup[s_1,\, T]$. The integral $\bar{J}_{21}(T)$ restricted to corresponding subregion is $\bar{J}_{211}(T)$ and $\bar{J}_{212}(T)$ respectively, where
				\begin{align*}
					\bar{J}_{211}(T)=&\int^T_1\dif s_1\int_0^{s_1-1}\dif t_1\int_0^{t_1}\dif s_2\int^{s_1}_{t_1+1}(s_1-s_2)^{2H-2}(t_2-t_1)^{2H-2}e^{-|t_1-s_1|-|t_2-s_2|}\dif t_2,\\
					\bar{J}_{212}(T)=&\int^T_1\dif s_1\int_0^{s_1-1}\dif t_1\int_0^{t_1}\dif s_2 \int_{s_1}^{T} (s_1-s_2)^{2H-2}(t_2-t_1)^{2H-2}e^{-|t_1-s_1|-|t_2-s_2|}\dif t_2.
				\end{align*}

For integral $\bar{J}_{211}(T)$, making the change of variables $u=t_2-t_1,\,v=s_1-t_1,\,x=s_1-s_2$, we have
				\begin{align*}
					\bar{J}_{211}(T)=\int_{1}^{T}\dif s_1\int^{s_1}_1 e^{-x}x^{2H-2}\dif x\int^x_1\dif v\int^v_1 e^{-u}u^{2H-2}\dif u.
				\end{align*}

The Partial integration formulate implies that when $T\to\infty$, the asymptote of $\bar{J}_{211}(T)$ is
				\begin{align}
& T \int^{\infty}_1e^{-x}x^{2H-2}\dif x\int^x_1 e^{-u}u^{2H-2} (x-u)\dif u+C_H.  \label{BarJ211 jianjinxian}
				\end{align}

For integral $\bar{J}_{211}(T)$, making the change of variables $u=t_2-t_1,\,v=s_1-t_1,\,x=s_1-s_2$, we have
				\begin{align*}
					\bar{J}_{212}(T)=\int^T_1\dif t_2\int^{t_2}_1e^{-y}\dif y\int^y_1 u^{2H-2}\dif u \int^u_1 e^{-v}(y-u+v)^{2H-2}\dif v.
				\end{align*}

The Partial integration formulate implies that when $T\to\infty$, the asymptote of $\bar{J}_{212}(T)$ is
				\begin{align}
& T\times\int^{\infty}_1e^{-y}\dif y\int^y_1 u^{2H-2}\dif u \int^u_1 e^{-v}(y-u+v)^{2H-2}\dif v +C_H\notag\\
&=T\times\Big[2\int_1^{\infty}e^{-u}u^{2H-2}\dif u \int_1^{u} e^{-v} v^{2H-1}\dif v  -\big(\int_1^{\infty} e^{-u}u^{2H-2}\dif u \big)^2 \Big]+C_H.\label{BarJ212 jianjinxian}
				\end{align}

Fix the integral variable $(s_1, s_2,t_1)$ of $\bar{J}_{22}(T)$ again. We decompose the integral region $[t_1+1,\,T]$ of integral variable $t_2$ into $[t_1+1,\,s_2+1]\cup[s_2+1,\, s_1]\cup [s_1,T]$. The integral $\bar{J}_{22}(T)$ restricted to corresponding subregion is $\bar{J}_{221}(T),\,\bar{J}_{222}(T),\,\bar{J}_{223}(T)$ respectively, where
				\begin{align*}
					\bar{J}_{221}(T)=&\int^T_1\dif s_1\int_0^{s_1-1}\dif s_2\int^{s_2}_0\dif t_1 \int_{t_1+1}^{s_2+1} (s_1-s_2)^{2H-2}(t_2-t_1)^{2H-2}e^{-|t_1-s_1|-|t_2-s_2|}\dif t_2,\\
					\bar{J}_{222}(T)=&\int^T_1\dif s_1\int_0^{s_1-1}\dif s_2\int^{s_2}_0\dif t_1 \int^{s_1}_{s_2+1} (s_1-s_2)^{2H-2}(t_2-t_1)^{2H-2}e^{-|t_1-s_1|-|t_2-s_2|}\dif t_2,\\
					\bar{J}_{223}(T)=&\int^T_1\dif s_1\int_0^{s_1-1}\dif s_2\int^{s_2}_0\dif t_1 \int^T_{s_1}(s_1-s_2)^{2H-2}(t_2-t_1)^{2H-2}e^{-|t_1-s_1|-|t_2-s_2|}\dif t_2.
				\end{align*}

For integral $\bar{J}_{221}(T)$ and integral $\bar{J}_{222}$, making the change of variables, we have
\begin{align*}
				\bar{J}_{221}(T)&=\int_{1}^{T}\dif s_1\int^{s_1}_1 e^{-v}\dif v\int^v_1(v-z+1)^{2H-2}\dif z\int^z_1 x^{2H-2}e^{-|z-x-1|}\dif x,\\
				\bar{J}_{222}(T)&=\int^T_1\dif s_1\int^{s_1}_1e^{-v}\dif v\int^v_1e^{-x}x^{2H-2}\dif x\int^x_1e^{z-1}(v-z+1)^{2H-2}\dif z.
\end{align*}

The Partial integration formulate and Fubini Theorem imply that when $T\to\infty$, the asymptotes of $\bar{J}_{221}(T)$ and  $\bar{J}_{222}(T)$ are
\begin{align}
& T\times \int^{\infty}_1e^{-v}\dif v\int^v_1(v-z+1)^{2H-2}\dif z\int^z_1 x^{2H-2}e^{-|z-x-1|}\dif x+C_H\notag\\
&=\frac{3}{2}T\times \Big(\int_1^{\infty} e^{-u} u^{2H-2}\dif u \Big)^2+ C_H\\
&T\times \int^{\infty}_1e^{-v}\dif v\int^v_1e^{-x}x^{2H-2}\dif x\int^x_1e^{z-1}(v-z+1)^{2H-2}\dif z+C_H\notag\\
&=2T\times\int_1^{\infty}e^{-x}x^{2H-2}\int_1^{x} e^{-y}(y^{2H-1}-y^{2H-2})\dif y  +C_H.
\end{align}

For integral $\bar{J}_{223}(T)$, making the change of variables, we have
				\begin{align*}
					\bar{J}_{223}(T)=\int_{1}^{T}\dif t_2\int^{t_2}_1e^{-u}u^{2H-2}\dif u\int^u_1\dif y\int^y_1e^{-x}x^{2H-2}\dif x.
				\end{align*}

By the Partial integration formulate,  we obtain when $T\to\infty$, the asymptotes of $\bar{J}_{223}(T)$ is 
\begin{align}
T\times \int^{\infty}_1  e^{-u}u^{2H-2}\dif u\int^u_1 e^{-x}x^{2H-2} (u-x)\dif x +C_H. \label{BarJ223 jianjinxian}
\end{align}
Combining  \eqref{BarJ211 jianjinxian}--\eqref{BarJ223 jianjinxian}, we obtain the asymptote of $\bar{J}_{2}(T)$
\begin{align}
 T\Bigg[ (4H-\frac{5}{2})\big(\int_1^{\infty} e^{-u} u^{2H-2}\dif u \big)^2 +2 \int_1^{\infty} e^{-1-u} u^{2H-2}\dif u   \Bigg]. \label{BarJ2 jianjinxian}
\end{align}
Finally, we obtain the asymptote \eqref{BarJ jianjinxian} of $\bar{J}(T)$ by combining \eqref{barJ1T  jianjx} and \eqref{BarJ2 jianjinxian}
			\end{proof}

\begin{lemma}\label{2.2}
Let \eqref{nT expression} be the expression of the quadruple integral $N(T)$, then when $T\to\infty$, the asymptote of $N(T)$ is 
\begin{align}
    T&\times \Bigg[\int_1^{\infty} e^{-u}u^{2H-2}\dif u\times \Big[e^{-1}-e   +(4H-1)\int_0^1(e^{x}-e^{-x})x^{2H-1}\dif x\Big] \notag\\
    &+\int_0^1(e^{x-1}-e^{-x-1})x^{2H-1}\dif x \Bigg] +C_H.\label{NT jianjinxian}
\end{align}
			\end{lemma}
			
\begin{proof}
We divide the integral region $\set{0\le s_1\le T,\,\, {(s_1-1)\vee 0}\le s_2\le {(s_1+1)\wedge T}}$	
of  integral variable $(s_1,\,s_2)$ of $N(T)$ into 
\begin{align*}
   \set{0\le s_1-1\le s_2\le s_1\le T}\cup\set{0\le s_2-1\le s_1\le s_2\le T}\\
   \cup\set{0\le s_2\le s_1\le 1}\cup\set{0\le s_1\le s_2\le 1}. 
\end{align*}

The integral $N(T)$ over the corresponding region is $N_1(T),\,N_2(T),N_3(T),N_4(T)$, where
				\begin{align*}
					N_1(T)=&\int_{[1,T]^2} e^{-|t_1-s_1|}\dif t_1 \dif s_1 \int^{t_1-1}_0(t_1-t_2)^{2H-2}\dif t_2 \int^{s_1}_{s_1-1} \\
				&\times\sgn(s_2-t_2)e^{-|t_2-s_2|}(s_1-s_2
					)^{2H-1}\dif s_2,\\
					N_2(T)=&-\int_{[1,T]^2}\dif t_1 \dif s_2\int^{t_1-1}_0(t_1-t_2)^{2H-2}\dif t_2 \int^{s_2}_{s_2-1}\\
				&\times\sgn(s_2-t_2)e^{-|t_1-s_1|-|t_2-s_2|}(s_2-s_1
					)^{2H-1}\dif s_1, \\
					N_3(T)=&\int^T_1\dif t_1\int^{t_1-1}_0(t_1-t_2)^{2H-2}\dif t_2 \int^1_0\dif s_1\int^{s_1}_0\\
				&\times\sgn(s_2-t_2)e^{-|t_1-s_1|-|t_2-s_2|}(s_1-s_2
					)^{2H-1}\dif s_2,\\
					N_4(T)=&\int^T_1\dif t_1\int^{t_1-1}_0(t_1-t_2)^{2H-2}\dif t_2\int^1_{0}\dif s_2\int^{s_2}_{0}\\
				&\times\sgn(s_2-t_2)e^{-|t_1-s_1|-|t_2-s_2|}(s_2-s_1
					)^{2H-1}\dif s_1. 
				\end{align*}
First, by the absolute integrability of the double integral		
$$\int^{\infty}_1e^{-t_1}\dif t_1\int^{t_1-1}_0 (t_1-t_2)^{2H-2}\dif t_2\int_{[0,1]^2}\abs{s_2-s_1
					}^{2H-1}e^{s_1} \dif s_1 \dif s_2  ,$$
we know the limit of $N_3(T),\,N_4(T)$ exists when $T\to\infty$. Therefore, integral $N(T)$ and integral $N_1(T)+N_2(T)$ have asymptotes with the same slope but different intercepts. Next we take the asymptotes of $N_1(T)$ and $N_2(T)$ respectively.
 
   We then should decompose the integral region $[1,T]^2$ of integral variable $(s_1, t_1)$ of $N_{1}(T)$ into $1\le t_1\le s_1\le T$ and $1\le s_1\le t_1\le T$ to take the asymptote of $N_1(T)$. And we have
				\begin{align*}
					N_{1}(T)=&(\int_1^T\dif s_1\int^{s_1}_1\dif t_1+\int_1^T\dif t_1\int^{t_1}_1\dif s_1)\int^{t_1-1}_0\dif t_2\int^{s_1}_{s_1-1}\\
					&\times(s_1-s_2)^{2H-1}(t_1-t_2)^{2H-2}\sgn(s_2-t_2)e^{-|t_1-s_1|-|t_2-s_2|}\dif s_2\\
					:=&N_{11}(T)+N_{12}(T)
				\end{align*}
For $N_{11}(T)$, making the change of variables $u= t_1-t_2$, $v =s_1-s_2$, we obtain
				\begin{align*}
					N_{11}(T)=\int_{0}^{1}e^{v}v^{2H-1}\dif v\int_{1}^{T}e^{-2s_1}\dif s_1\int_{1}^{s_1}e^{2t_1}\dif t_1\int_{1}^{t_1}e^{-u}u^{2H-2}\dif u.
				\end{align*}
Therefore, by the Partial integration formulate,  we obtain when $T\to\infty$, the asymptote of $N_{11}(T)$ is 
\begin{align}
\frac{T}{2}\int_{0}^{1}e^{v}v^{2H-1}\dif v\int^\infty_1e^{-u}u^{2H-2}\dif u +C_H.\label{N11T jianjinxian}
				\end{align}
For $N_{12}(T)$, we  fix integral variable $t_1$ and decompose  the integral region $[1,\,t_1]\times [0,t_1-1]$ of integral variable $(s_1, t_2)$  into $0\le t_2\le s_1-1\le t_1-1$ and $1\le s_1\le t_2+1\le t_1$. Then $N_{12}(T)$ split into the sum of the following two integrals:
				\begin{align*}
					N_{12}(T)=&\int_{1}^{T}\dif t_1(\int_{1}^{t_1}\dif s_1\int_{0}^{s_1-1}\dif t_2+\int_{0}^{t_1-1}\dif t_2\int_1^{t_2+1}\dif s_1)\int^{s_1}_{s_1-1}\\
					&\times(s_1-s_2)^{2H-1}(t_1-t_2)^{2H-2}\sgn(s_2-t_2)e^{s_1-t_1-|t_2-s_2|}\dif s_2\\
					:=&O_1(T)+O_2(T)
				\end{align*}
For $O_1(T)$, making the change of variables $u=t_1-t_2,\,v=s_1-t_2,\,x= s_1-s_2$, we have
				\begin{align*}
					O_1(T)=\int^1_0e^xx^{2H-1}\dif x\int_{1}^{T}\dif t_1\int_{1}^{t_1}e^{-u}u^{2H-2}(u-1)\dif u. 
				\end{align*}
Therefore, by the Partial integration formulate,  when $T\to\infty$, the asymptotes of $O_{1}(T)$ is 
				\begin{align}
				T\times \int^1_0e^x x^{2H-1}\dif x\int_{1}^{\infty}e^{-u}(u^{2H-1}-u^{2H-2})\dif u+ C_H.\label{O1 jianjinxian}
				\end{align}
For $O_2(T)$,  making the change of variables $u=t_1-t_2,\,v=t_1-s_1+1,\,x=s_1-s_2$ indicates
				\begin{align*}
					O_2(T)=&\int_{1}^{T}\dif t_1\int^{t_1}_{1}e^{1-v}\dif v\int_{1}^{v}u^{2H-2}\dif u\int^1_0x^{2H-1}\\
					&\times\sgn(u-v-x+1)e^{-|u-v-x+1|}\dif x.
				\end{align*}
By the Partial integration formulate,  when $T\to\infty$, the asymptote of $O_{2}(T)$ is 
				\begin{align}
				&{T} \int^{\infty}_{1}e^{1-v}\dif v\int_{1}^{v}u^{2H-2}\dif u\int^1_0 x^{2H-1} \sgn(u-v-x+1)e^{-|u-v-x+1|}\dif x + C_H\notag\\
				&=
				 {T} \int_1^{\infty}e^{-u}u^{2H-2}\dif u\int_0^1 e^x \big(\frac12 x^{2H-1}-x^{2H}\big)\dif x + C_H, \label{O2 jianjinxian}
				\end{align}
The last equation is obtained by decomposing the region $\set{1\le u\le v<\infty}$ into $\set{1\le u\le v\le 1+u<\infty}\cup\set{1\le u\le v-1<\infty}$, using the Fubini theorem. Combining \eqref{O1 jianjinxian} and \eqref{O2 jianjinxian}, we obtain the asymptote of $N_{12}(T)$ is:
\begin{align}
T\Big[\int^1_0e^{x-1} x^{2H-1}\dif x -\int_{1}^{\infty} e^{1-u} u^{2H-2}\dif u +  (4H-\frac32)\int_{1}^{\infty} e^{-u} u^{2H-2}\dif u\int_0^1 e^x x^{2H-1}\dif x\Big]+C_H. \label{N12T jianjinxian}
\end{align}
Combining the equation and the asymptote \eqref{N11T jianjinxian}, the asymptote of $N_{1}(T)$ is 
\begin{align}
T\Big[\int^1_0e^{x-1} x^{2H-1}\dif x -\int_{1}^{\infty} e^{1-u} u^{2H-2}\dif u +  (4H-1)\int_{1}^{\infty} e^{-u} u^{2H-2}\dif u\int_0^1 e^x x^{2H-1}\dif x\Big]+C_H. \label{N1T jianjinxian}
\end{align}
To take the asymptote of $-N_{2}(T)$, we decompose  the integral region $[1,T]^2$ of integral variable $(t_1,s_2)$ of $-N_{2}(T)$ into $\set{1\le t_1\le s_2\le T}\cup\set{1\le s_2\le t_1\le T}$. Then we obtain
				\begin{align*}
					-N_{2}(T)=&\big(\int^T_1\dif s_2\int^{s_2}_1\dif t_1+\int^T_1\dif t_1\int^{t_1}_1\dif s_2\big)\int^{t_1-1}_0\dif t_2\int^{s_2}_{s_2-1}\\	
					&\times(s_2-s_1)^{2H-1}(t_1-t_2)^{2H-2}\sgn(s_2-t_2)e^{-|t_1-s_1|-|t_2-s_2|}\dif s_1\\  
					:=&N_{21}(T)+N_{22}(T).
				\end{align*}
Making the change of variables, we have
				\begin{align*}
					N_{21}(T)=\int_{1}^{T}\dif s_2\int_{1}^{s_2}e^{-v}\dif v\int^{v}_1u^{2H-2}\dif u\int^{1}_{0}x^{2H-1}e^{-|x-v+u|}\dif x.
				\end{align*}
 The Partial integration formulate and making the change of variable $z=v-u$ imply that when $T\to\infty$, the asymptote of $N_{21}(T)$ is 
				\begin{align}
 &T\times \int^\infty_{1}u^{2H-2}\dif u \int_{u}^{\infty} e^{-v}\dif v\int^{1}_{0}x^{2H-1}e^{-|x-v+u|}\dif x+C_H\notag\\
 &= {T} \times \big[(2H+\frac{1}{2})\int_0^1 e^{-x}x^{2H-1}\dif x -e^{-1}\big]\times \int_1^{\infty} e^{-u} u^{2H-2}\dif u+C_H. \label{N21T jianjinxian}
				\end{align}
For $N_{22}(T)$, we  fix integral variable $t_1$ and decompose  the integral region $ [0,t_1-1]\times[1,\,t_1]$ of integral variable $(t_2,s_2)$  into $1\le t_2+1\le s_2\le t_1$ and $1\le s_2\le t_2+1\le t_1$. Then $N_{22}(T)$ split into the sum of the following two integrals:
				\begin{align*}
					N_{22}(T)=&\int_{1}^{T}\dif t_1\big(\int^{t_1-1}_0\dif t_2\int^{t_1}_{t_2+1}\dif s_2+\int^{t_1}_1\dif s_2\int^{t_1-1}_{s_2-1}\dif t_2\big)\int_{s_2-1}^{s_2}\\
					&\times(s_2-s_1)^{2H-1}(t_1-t_2)^{2H-2}\sgn(s_2-t_2)e^{-|t_1-s_1|-|t_2-s_2|}\dif s_1\\
					:=&O'_1(T)+O'_2(T).
				\end{align*}
Making the change of variables $u=t_1-t_2,\, v= t_1+1-s_2,\,x=s_2-s_1$, we have
				\begin{align*}
				O'_1(T)&=\int^T_1\dif t_1\int^{t_1}_1e^{-u}u^{2H-2}\dif u\int_{1}^{u}\dif v\int^1_0e^{-x}x^{2H-1}\dif x,\\
				O_2'(T)&=\int_{1}^{T}\dif t_1\int_{1}^{t_1}e^{-v}\dif v\int_{1}^{v}e^{-|u-v+1|}u^{2H-2}\sgn(u-v+1)\dif u\int^1_0 e^{1-x}x^{2H-1}\dif x.
				\end{align*}
The Partial integration formulate implies that when $T\to\infty$, the asymptotes of $O'_1(T)$ and $O_2'(T)$ are 
				\begin{align}
		&T\times\int^1_0e^{-x}x^{2H-1}\dif x \int^\infty_1e^{-u}u^{2H-2}(u-1)\dif u+C_H, \label{O'1T jianjinxian} \\
		&T\times \int^1_0e^{1-x}x^{2H-1}\dif x\int^\infty_1e^{-v}\dif v\int_{1}^{v}e^{-|u-v+1|}u^{2H-2}\sgn(u-v+1)\dif u+C_H\notag\\
		&=\frac{T}{2}\times\int^1_0e^{-x}x^{2H-1}\dif x\times \int^{\infty}_1e^{-u}u^{2H-2}\dif u  +C_H. \label{O2'T jianjinxian}
				\end{align}
Combining \eqref{O'1T jianjinxian} and \eqref{O2'T jianjinxian}, we obtain the asymptote of $N_{22}(T)$ is
\begin{align}
 T\times\int^1_0e^{-x}x^{2H-1}\dif x \int^\infty_1e^{-u}u^{2H-2}(u-\frac12)\dif u+C_H.  \label{N22T jianjinxian}.
\end{align}
Combining \eqref{N21T jianjinxian} and \eqref{N22T jianjinxian},  the asymptote of $N_{2}(T)$ is
\begin{align*}
T\Big[\int^1_0 e^{-x-1} x^{2H-1}\dif x -\int_{1}^{\infty} e^{-1-u} u^{2H-2}\dif u +  (4H-1)\int_{1}^{\infty} e^{-u} u^{2H-2}\dif u\int_0^1 e^{-x} x^{2H-1}\dif x\Big]+C_H. 
\end{align*}
The asymptote of $N(t)$ is \eqref{NT jianjinxian}, which is obtained by  the asymptote \eqref{N1T jianjinxian} of $N_1(t)$ minus the above equation.
			\end{proof}
			
\begin{lemma}\label{m31.2.3}
Let \eqref{bar nT expression} be the expression of the quadruple integral $\tilde{N}(T)$, then when $T\to\infty$, the asymptote of $\tilde{N}(T)$ is
\begin{align}
 T\times\Big[(1+e^{-2})+\big[(2H+1)e^{-1}+(2H-1) e\big] \int_1^{\infty} e^{-u}u^{2H-2}\dif u \Big] + C_H.\label{nt jianjx bds}
\end{align}
\end{lemma}
\begin{proof} By integrating Dirac function, we write the quadruple integral $\tilde{N}(T)$ as the sum of the following two triple integrals:
\begin{align}
    \tilde{N}(T)&=\tilde{N}_1(T)+ \tilde{N}_2(T),\label{barn1barn2fenjie}
\end{align} where
\begin{align*}
    \tilde{N}_1(T)&= \int^T_0\dif s_1\int^T_1 e^{-|t_1-s_1|}\dif t_1\int^{t_1-1}_0 (t_1-t_2)^{2H-2} \dif t_2 \big(s_1-(s_1-1)\vee 0 \big)^{2H-1} e^{-|t_2-(s_1-1)\vee 0|}, \\
		\tilde{N}_2(T)&=\int^T_0\dif s_1\int^T_1 e^{-|t_1-s_1|}\dif t_1\int^{t_1-1}_0 (t_1-t_2)^{2H-2} \dif t_2 \big((s_1+1)\wedge T-s_1\big)^{2H-1} e^{-|t_2-(s_1+1)\wedge T|}.
\end{align*}

First, solve the asymptote of triple integral $\tilde{N}_1(T)$. Then divide the integral region $s_1\in [0,T]$ into $[0,1)\cup [1,T]$. Making the change of variable $u=t_1-t_2$, we get that the triple integral of the subinterval $s_1\in [0,1)$ connection is:
\begin{align*}
   & \int^1_0 s_1 ^{2H-1} \dif s_1\int^T_1 e^{-(t_1-s_1)}\dif t_1\int^{t_1-1}_0  e^{-t_2 }(t_1-t_2)^{2H-2} \dif t_2\\
    &=\int^1_0 e^s s_1 ^{2H-1} \dif s_1\int^T_1 e^{-2t_1 }\dif t_1\int^{t_1}_1  e^{u}u^{2H-2} \dif u.
\end{align*} When $T\to \infty$, its limit exists. Then the asymptote of $\tilde{N}_1(T)$ and the triple integral
\begin{align*}
   \tilde{N}_{11}(T)&= \int^T_1  \dif s_1\int^T_1 e^{-\abs{t_1-s_1}}\dif t_1\int^{t_1-1}_0  e^{-\abs{t_2-s_1+1 }}(t_1-t_2)^{2H-2} \dif t_2
\end{align*}
connected with $\tilde{N}_1(T)$ in the integral sub region $s_1\in [1,T]$ have the same slope asymptote (different intercept terms). Making the change of variables $$w=s_1\vee t_1,\,\,v=\abs{s_1-t_1},\,\, u=t_1-t_2$$
Triple integral $\tilde{N}_{11}(T)$ is rewritten as
\begin{align}
    \tilde{N}_{11}(T)&=\int^T_1  \dif w \int^{w-1}_0 e^{-v}\dif v\Big[\int_1^{w-v}u^{2H-2}e^{-(v+u-1)}\dif u + \int_1^{w}u^{2H-2}e^{-\abs{v-u+1}}\dif u\Big].\label{n11t biaodsh}
\end{align}
Making the change of variables $$u'=u-1,\,\,x=v+u' $$
We know that the first part of triple integral \eqref{n11t biaodsh} is
\begin{align*}
& \int^T_1  \dif w \int^{w-1}_0 e^{-v}\dif v \int_1^{w-v}u^{2H-2}e^{-(v+u-1)}\dif u \\
&=\int^T_1  \dif w \int^{w-1}_0 e^{-2x}\dif x \int_0^{x}(1+u')^{2H-2}e^{u'}\dif u'.
\end{align*}
It is easy to know from the Partial integral formula and Fubini theorem that the asymptote of the above formula is
\begin{align}
T \int^{\infty}_0 e^{-2x}\dif x \int_0^{x}(1+u')^{2H-2}e^{u'}\dif u' +C_H =\frac12 T\int_1^{\infty} e^{1-u}u^{2H-2}\dif u +C_H. \label{diyibufenjianjinxian}
\end{align}
Making the change of variable $u'=u-1$ and Fubini theorem, the second part of triple integral \eqref{n11t biaodsh} is
\begin{align*}
    &\int^T_1  \dif w \int^{w-1}_0 e^{-v}\dif v\int_1^{w}u^{2H-2}e^{-\abs{v-u+1}}\dif u\\
    &=\int^T_1  \dif w \int^{w-1}_0 e^{-v}\dif v\Big[\int_1^{1+v}u^{2H-2}e^{ u-v-1}\dif u+\int_{1+v}^{w}u^{2H-2}e^{  1+v-u }\dif u\Big]\\
    &=\int^T_1  \dif w \int^{w-1}_0 e^{-2v}\dif v\int_0^{v} e^{ u' } (u'+1)^{2H-2}\dif u'+\int^T_1  \dif w \int^{w-1}_0 e^{ -u'} (u'+1)^{2H-2}u'\dif u'
\end{align*}
According to the Partial integral formula, the asymptote of the above formula is
\begin{align}
   &T \Big[\int^{\infty}_0 e^{-2x}\dif x \int_0^{x}(1+u')^{2H-2}e^{u'}\dif u'+\int^{\infty}_0 e^{-u'} (1+u')^{2H-2}u'\dif u' \Big]+C_H \notag\\
   &=T \Big[\frac12\int_1^{\infty} e^{1-u}u^{2H-2}\dif u+\int^{\infty}_1 e^{1-u} u^{2H-2}(u-1)\dif u \Big]+C_H.\label{dierbufen jianjinxian}
\end{align}
Combining \eqref{diyibufenjianjinxian} and \eqref{dierbufen jianjinxian}, we get the asymptote of triple integral $\tilde{N}_1(T)$ as:
\begin{align}
    T \times \int^{\infty}_1 e^{1-u} u^{2H-1} \dif u +C_H.\label{barn11t jianjinxian}
\end{align}

Next, we solve the asymptote of triple integral $\tilde{N}_2(T)$. Similarly, we divide the integral region $s_1\in [0,T]$ into $[0,T-1]\cup (T-1,T]$. Making the change of variable $u=t_1-t_2$, we obtain the limit existence of triple integrals associated with subinterval $s_1\in (T-1,T]$ when $T\to \infty$. Therefore, the asymptote of $\tilde{N}_2(T)$ and the following triple integral
\begin{align*}
    \tilde{N}_{21}(T)&=\int^{T-1}_{0}  \dif s_1\int^T_1 e^{-\abs{t_1-s_1}}\dif t_1\int^{t_1-1}_0  e^{-\abs{t_2-s_1-1 }}(t_1-t_2)^{2H-2} \dif t_2
\end{align*}
have the same slope (different intercept terms). For triple integrals $\tilde{N}_{21}(T)$, first making the change of variable $u=t_1-t_2$, and then we divide the integral region $[0,T-1]\times [1,T]$ of the integral variable $(s_1,t_1)$ as follows:
\begin{align*}
  \set{ 1\le t_1\le s_1\le T-1 }\cup \set{ 0\le t_1-1\le s_1\le t_1\wedge(T-1)\le T}\cup\set{ 0\le s_1\le t_1-1\le T-1 }, 
\end{align*}
we have
\begin{align}
    \tilde{N}_{21}(T)&=\int^{T-1}_{0}  \dif s_1\int^T_1 e^{-\abs{t_1-s_1}}\dif t_1\int^{t_1}_1  e^{-\abs{t_1-s_1-u-1 }}u^{2H-2} \dif u\notag \\
    &= \big[\int^{T-1}_{1}\dif s_1\int^{s_1}_1\dif t_1+\int_1^T\dif t_1\int^{t_1\wedge (T-1)}_{t_1-1}\dif s_1+\int_1^T\dif t_1\int_{0}^{t_1-1}\dif s_1\big] \notag\\
    &\times \int^{t_1}_1  e^{-\abs{t_1-s_1}-\abs{t_1-s_1-u-1 }}u^{2H-2} \dif u.\label{n21t biaodashi}
\end{align}
According to the Partial integral formula, the asymptotes of the first, second and third parts of triple integral \eqref{n21t biaodashi} are:
\begin{align*}
     \frac{T}{2}  \times \int^{\infty}_1 e^{-1-u} u^{2H-2} \dif u +C_H,\\
    T\times \int^{\infty}_1 e^{-1-u} u^{2H-2} \dif u +C_H,\\
    T\times  \int^{\infty}_1 e^{-1-u} \big[u^{2H-1}+\frac12u^{2H-2} \big] +C_H.
\end{align*}
Combining the three asymptotes above, we get triple integral the asymptote of $\tilde{N}_2(T)$ is
\begin{align}
    T\times  \int^{\infty}_1 e^{-1-u} \big[u^{2H-1}+ 2 u^{2H-2} \big] +C_H.\label{n2t jianjinxian biaods}
\end{align}

Finally, we combine the asymptote \eqref{barn11t jianjinxian} of $\tilde{N}_1(T)$ and the asymptote \eqref{n2t jianjinxian biaods} of $\tilde{N}_2(T)$ to obtain the asymptote \eqref{nt jianjx bds} of $\tilde{N}(T)$.
\end{proof}	

\begin{lemma}\label{UT BarUT jianjx yinli}
Record two quadruple integrals $U(T),\,\tilde{U}(T)$, as given in \eqref{UT expression} and \eqref{bar UT expression}. When $T\to\infty$, their asymptotes are:
	\begin{align}
				T&\times \Bigg[\int_1^{\infty} e^{-u}u^{2H-2}\dif u\times \Big[e^{-1}-e  +{(4H-1)}\int_0^1(e^{x}-e^{-x})x^{2H-1}\dif x\Big] \notag\\
    &+\int_0^1(e^{x-1}-e^{-x-1})x^{2H-1}\dif x \Bigg] +C_H,\label{UT jianjinxian}\\
				T& \times\Big[(1+e^{-2})+\big[(2H+1)e^{-1}+(2H-1) e\big] \int_1^{\infty} e^{-u}u^{2H-2}\dif u \Big]	 . \label{BarUT jianjinxian}
				\end{align} 
\end{lemma}
The proof of Lemma \ref{UT BarUT jianjx yinli} is basically consistent with the proof of Lemma \ref{2.2} and Lemma \ref{m31.2.3} above. Considering the length of the article, its details are omitted.
			\begin{lemma}\label{yyuhjkl}
			    The marked quadruple integral $L(T)$ is given by \eqref{LT biaodashi}. Then the asymptote of $L(T)$ when $T\to\infty$ is:
				\begin{align}
				 4T&\times\Bigg[(4H+1)\int_0^1e^{-u}u^{2H-1}\dif u\int_0^{u}e^v v^{2H-1}\dif v   -(2H+\frac{1}{2})\big(\int_0^1 e^{-u}u^{2H-1}\dif u\big)^2\notag\\
				 &+ \int_0^1 (e^{-u-1}-e^{u-1})u^{2H-1} \dif u \Bigg]+ C_H.
					 \label{lt jianjinxian}
				\end{align} 
			\end{lemma}
			\begin{proof}
			    The starting point is to remove the following two absolute value symbols from the quadruple integral $L(T)$: $|s_1-s_2|$ $|s_1-s_2|$ and $|t_1-t_2|$. That is, first of all, we divide the integral region of the integral variables $s_2,t_2$ as follows:
				\begin{align}
					\int^{(s_1+1)\wedge T}_{(s_1-1)\vee 0}\dif s_2\int^{(t_1+1)\wedge T}_{(t_1-1)\vee 0}\dif t_2=(\int^{s_1}_{(s_1-1)\vee0}+\int_{s_1}^{(s_1+1)\wedge T})\dif s_2(\int^{t_1}_{(t_1-1)\vee0}+\int_{t_1}^{(t_1+1)\wedge T})\dif t_2.\label{sigefenjie}
				\end{align}
				The integral values of the four integral $L(T)$ in the above four sub regions are recorded as L$L_1(T),L_2(T),L_3(T),L_4(T)$. It is easy to know by symmetry:
				\begin{align}
					L_1(T)=L_4(T),\qquad L_2(T)=L_3(T).\label{lgdengshi}
				\end{align} 
				
				Then remove the two symbols $\vee$ in $L_1(T)$ respectively, and further decompose the integral into the sum of the following four integrals:
				\begin{align*}
					L_1(T)=&(\int^T_1\dif s_1\int^{s_1}_{s_1-1}\dif s_2+\int^1_0\dif s_1\int^{s_1}_0\dif s_2)(\int^T_1\dif t_1\int^{t_1}_{t_1-1}+\int^1_0\dif t_1\int^{t_1}_0)\\
					&\times(s_1-s_2)^{2H-1}(t_1-t_2)^{2H-1}e^{-|t_1-s_1|-|t_2-s_2|}\dif t_2\\
					:=&L_{11}(T)+L_{12}(T)+L_{13}(T)+L_{14}(T)
				\end{align*}
		First, we notice that the integral $L_{14}(T)$ is independent of $T$, and we get $L_{12}(T)=L_{13}(T)$ through symmetry. Making the change of variables $u=s_1-s_2,\,v=t_1-t_2$, we can deduce:
		\begin{align*}
				  L_{12}(T)&=\int_1^T e^{-s_1}\dif s_1\int_0^1 u^{2H-1}\dif u\int_0^1 e^{t_1}\dif t_1\int_0^{t_1} e^{-\abs{s-t+v-u}}v^{2H-1}\dif v,
				\end{align*}
		Therefore, the $\lim_{T\to\infty}L_{12}(T)$ exists. Again according to symmetry and making the change of variables $u=s_1-s_2,\, v=t_1-t_2,\,x=s_1-t_1$, it is obtained that:
				\begin{align*}
					L_{11}(T)
				&=2\int^T_1\dif s_1\int^{s_1-1}_0 e^{-x}\dif x\int^1_0u^{2H-1}\dif u\int^1_0e^{-|x-u+v|}v^{2H-1}\dif v
				\end{align*}
	It can be seen from the Partial integration formula that the asymptotes of $L_{11}(T)$ and $L_{1}(T)$ when $T\to\infty$ are both (only the intercept term $C_H$ with different difference):
				\begin{align}				&2T\int^{\infty}_{0}e^{-x}\dif x\int^1_0u^{2H-1}\dif u\int^1_0e^{-|x-u+v|}v^{2H-1}\dif v +C_H\notag \\
				&= 2T\Big[(4H+1)\int_0^1e^{-u}u^{2H-1}\dif u\int_0^{u}e^v v^{2H-1}\dif v -\int_0^1e^{v-1}v^{2H-1} \dif v\Big]  +C_H\label{l11 jianjinxian}
				\end{align}
				The above equation divides the integral region of $(u,v)$ into $\set{0\le u\le v\le 1}\cup\set{0\le v\le u\le 1}$, for the second sub region, we divide the integral region of $x$ into $[0,u-v]\cup(u-v,\infty)$ again, and then get it from Fubini theorem.
				
				Similarly, it is decomposed as follows:
				\begin{align*}
					L_2(T)=&-(\int^T_1\dif s_1\int^{s_1}_{s_1-1}\dif s_2+\int^1_0\dif s_1\int^{s_1}_0\dif s_2)(\int^T_1\dif t_2\int^{t_2}_{t_2-1}+\int^1_0\dif t_2\int^{t_2}_0)\\
				&\times(s_1-s_2)^{2H-1}(t_2-t_1)^{2H-1}e^{-|t_1-s_1|-|t_2-s_2|}\dif t_1\\
					:=&-(L_{21}(T)+L_{22}(T)+L_{23}(T)+L_{24}(T)),
				\end{align*}
Where the integral $L_{24}(T)$ is independent of $T$, and the existence of $\lim_{T\to\infty}\big(L_{22}(T)+L_{23}(T)\big)$ is deduced by making the change of variables $v=t_1-t_2,\,u=s_1-s_2$. Then, it can be seen from making the change of variables $$w=\max\set{s_1,\,t_2},\, x=\abs{s_1-t_2},\,u=s_1-s_2,\,v=t_2-t_1$$ that the quadruple integral $L_{21}(T)$ is:
				\begin{align*}
					L_{21}(T)
					&=2\int^{T}_1\dif w\int_{0}^{w}e^{-x}\dif x\int_{[0,1]^2}   e^{-|x-v|-u} u^{2H-1}v^{2H-1}\dif u \dif v.
				\end{align*}
It is known from the Partial integration formula that the asymptotes of $-L_{21}(T)$ and $L_{2}(T)$ when $T\to\infty$ are both (only different intercept terms $C_H$):
				\begin{align}
				&-2T \int_0^1 e^{-u} u^{2H-1} \dif u \times \int_0^{\infty} e^{-x} \dif x\int_0^1 e^{-\abs{x-v}}v^{2H-1} \dif v +C_H\notag\\
				&=-2T\int_0^1 e^{-u}u^{2H-1}\dif u \int_0^1 e^{-v}(v^{2H}+\frac12 v^{2H-1})\dif v +C_H,\label{l21 jianjinxian}
				\end{align}
				The above equation decomposes the integral domain of $(x,v)$ into $\set{0\le x\le v\le 1}\cup\set{0\le v\le 1,\, x>v}$, and then obtains it from Fubini theorem.
		
		Finally, from \eqref{sigefenjie} and \eqref{lgdengshi}, combining the asymptotes \eqref{l11 jianjinxian} and \eqref{l21 jianjinxian} of $L_{1}(T)$ and $L_{2}(T)$ when $T\to\infty$, it is obtained that the asymptote of $L(T)$ is \eqref{lt jianjinxian}.
			\end{proof}
	
\begin{lemma}\label{yyuhjkl02}
The quadruple integral $P(T)$ is given in equation \eqref{PT biaodashi}. Then the asymptote of $P(T)$ when $T\to\infty$ is:
				\begin{align}
 { 2T \times\Big[1-e^{-2} -(2H+1) \int_0^1 (e^{u-1} -e^{-u-1})u^{2H-1}\dif u\Big] +C_H}.\label{p jianjiexian}
				\end{align}
			\end{lemma}
\begin{proof}
First, by integrating Dirac function, we write $P(T)$ as the sum of the following two triple integrals:
\begin{align*}
    P(T)&=\int_{[0,T]^2} e^{-|t_1-s_1|}\dif s_1 \dif t_1\int^{(t_1+1)\wedge T}_{(t_1-1)\vee0} \sgn(t_1-t_2)|t_1-t_2|^{2H-1}\dif t_2  \notag\\
			& \times\Big[e^{-|t_2-(s_1-1)\vee0|}\sgn\big((s_1-1)\vee0-t_2\big) \big(s_1-(s_1-1)\vee0\big)^{2H-1}\notag\\
			&+e^{-|t_2-(s_1+1)\wedge T|}\sgn\big((s_1+1)\wedge T-t_2\big) \big((s_1+1)\wedge T-s_1\big)^{2H-1} \Big]\notag\\
			&:= P_1(T)+ P_2(T).
\end{align*}

It is required to solve the asymptote of triple integral $P_1(T)$. First, we divide the region $\set{(s_1,t_1)\in [0,T]^2}$ into the following parts:
\begin{align*}
    &[0,3]\times [0,1],\, [3,T]\times [0,1],\,[0,1]\times [1,T],\, [1,T-1]\times[T-1,T],\\
    &[T-1,T]\times[1,T-1],\,[T-1,T]^2,\,  [1,T-1]^2.
\end{align*}
It is clear that the triple integral $P_1(T)$ restricted in sub-region $[0,3]\times [0,1]$ is independent of $T$ and when $T\to\infty$, the limit of triple integral $P_1(T)$ in sub-region $[3,T]\times [0,1]$ exists. Making the change of variables $u=t_2-t_1$, we have that when $T\to\infty$, the limit of triple integral $P_1(T)$ in sub-region $[0,1]\times [1,T]$ exists; Making the change of variables $y=T-t_1,\,u=t_2-t_1$, when $T\to\infty$, the limit of triple integral $P_1(T)$ in sub region $[1,T-1]\times[T-1,T]$ exists.
It can be seen from making the change of variables $$x=T-s_1,\,y=T-t_1,\,u=t_2-t_1$$ that the integral of triple integral $P_1(T)$ in the sub region $[T-1,T]^2$ is also independent of $T$; When $T\to\infty$, the triple integral $P_1(T)$ is in the sub region $[T-1,T]\times [1,T-1]$ the limit of the integral exists.

Making the change of variables $$w=\max\set{s_1,\,t_1},\, v=\abs{s_1-t_1},\,u=t_2-t_1$$ we can see that the triple integral $P_1(T)$ in sub region $ [1,T-1]^2$ is:
\begin{align*}
   & \int_{[1,T-1]^2} e^{-|t_1-s_1|}\dif s_1 \dif t_1\int^{t_1+1 }_{t_1-1 } \sgn(t_1-t_2)|t_1-t_2|^{2H-1} e^{-|t_2-(s_1-1) |}\sgn\big(s_1-1 -t_2\big) \dif t_2 \\
&= \int_1^{T-1}\dif w \int_0^{w-1} e^{-v}\dif v\int_{-1}^1 \big[e^{-\abs{v+u+1}}+e^{-\abs{u-v+1}}\sgn(u-v+1)\big]\abs{u}^{2H-1}\sgn(u)\dif u.
\end{align*}
According to the Partial integration formula, When $T\to\infty$, the asymptote of the above formula and the asymptote of triple integral $P_1(T)$ are both (combined with the limit existence of integrals in the first six regions, it can be seen that they are only different from each other in terms of intercept $C_H$):
\begin{align*}
    &T \times \int_0^{\infty}e^{-v}\dif v\int_{-1}^1 \big[e^{-\abs{v+u+1}}+e^{-\abs{u-v+1}}\sgn(u-v+1)\big]\abs{u}^{2H-1}\sgn(u)\dif u + C_H\notag\\
    &=T\times \int_{-1}^{1}e^{-u-1} \abs{u}^{2H-1}\mathrm{sgn}(u)(1+u)\dif u + C_H. 
\end{align*}

Similarly, the asymptote of triple integral $P_2(T)$ has the same slope as the asymptote of the following integral \begin{align*}
& \int_{[1,T-1]^2} e^{-|t_1-s_1|}\dif s_1 \dif t_1\int^{t_1+1 }_{t_1-1 } \sgn(t_1-t_2)|t_1-t_2|^{2H-1} e^{-|t_2-(s_1+1) |}\sgn\big(s_1+1 -t_2\big) \dif t_2 \\   &=\int_1^{T-1}\dif w \int_0^{w-1} e^{-v}\dif v\int_{-1}^1 \big[-e^{-\abs{u-v-1}}+e^{-\abs{u+v-1}}\sgn(u+v-1)\big]\abs{u}^{2H-1}\sgn(u)\dif u\end{align*}
Thus, we get that the asymptote of the triple integral $P_2(T)$ is:
\begin{align*} 
&T\times \int_0^{\infty} e^{-v}\dif v\int_{-1}^1 \big[-e^{-\abs{u-v-1}}+e^{-\abs{u+v-1}}\sgn(u+v-1)\big]\abs{u}^{2H-1}\sgn(u)\dif u\\
&=T\times\int_{-1}^{1}e^{u-1} \abs{u}^{2H-1}\mathrm{sgn}(u)(u-1)\dif u  + C_H.
\end{align*}

Finally, we combine the two asymptotes of $P_1(T)$ and $P_2(T)$ to obtain the asymptote \eqref{p jianjiexian} of $P(T)$.
\end{proof}
 


\end{document}